\documentclass{cmslatex}
\usepackage[paperwidth=7in, paperheight=10in, margin=.875in]{geometry}

\usepackage{latexsym, amssymb, enumerate, amsmath}
\usepackage{IEEEtrantools}
\usepackage{xspace,bm}
\usepackage{framed,paralist,fixltx2e,xcolor,fancybox}
\usepackage{mathtools}%AMS extensions
\usepackage[normalem]{ulem}
\usepackage{tabu}
\usepackage[pdfborder={0 0 .1},breaklinks,pageanchor=false]{hyperref}% put this last...
\usepackage[initials,nobysame]{amsrefs}% Load after hyperref

\providecommand{\email}[1]{\texttt{#1}}

\providecommand\MoveEqLeft[1][2]{\kern #1em  &   \kern -#1em}

\DeclareRobustCommand{\qed}{%
  \ifmmode \mathqed
  \else
    \leavevmode\unskip\penalty9999 \hbox{}\nobreak\hfill
    \quad\hbox{\qedsymbol}%
  \fi
}
\numberwithin{equation}{section}

\makeatletter
\@mparswitchfalse
\makeatother
\newcommand{\sidenote}[1]{%
  \marginpar[\raggedleft\tiny #1]{\raggedright\tiny #1}
}
\newcommand{\Remove}[2][red]{\textcolor{#1}{\sout{#2}}}
\newcommand{\remove}[2][red]{%
  \ifmmode\text{\Remove[#1]{\ensuremath{#2}}}%
  \else\Remove[#1]{#2}\fi}
\newcommand{\Add}[2][blue]{\textcolor{#1}{\uwave{#2}}}
\newcommand{\add}[2][blue]{%
  \ifmmode\text{\Add[#1]{\ensuremath{#2}}}%
  \else\Add[#1]{#2}\fi}%

\definecolor{hlcolor}{Hsb}{60,.5,1}
\definecolor{shadecolor}{named}{hlcolor}
\newlength{\hbwidth}
\newcommand{\hbcolor}{hlcolor}

\newcommand{\Highlight}[2][hlcolor]{\bgroup\markoverwith{\textcolor{#1}{\rule[-.3em]{1pt}{1.2em}}}\ULon{#2}}
\newcommand{\highlight}[2][hlcolor]{%
  \ifmmode\text{\Highlight[#1]{\ensuremath{#2}}}%
  \else\Highlight[#1]{#2}\fi}%

\newenvironment{beqn}[1][;C?s]%
    {\left\{\begin{IEEEeqnarraybox}[\relax][c]{#1}}%
    {\end{IEEEeqnarraybox}\right.}
\newcommand{\del}{\partial}

\newcommand{\lap}{\Delta}

\newcommand{\inv}{^{-1}}

\newcommand{\gradperp}{\nabla^\perp}
\newcommand{\gradinv}{\nabla^{-1}}
\newcommand{\dv}{\nabla \cdot}
\newcommand{\curl}{\nabla \times}

\newcommand{\BS}{K_{BS}*}

\newcommand{\eps}{\varepsilon}
\renewcommand{\epsilon}{\eps}
\renewcommand{\leq}{\leqslant}

\renewcommand{\geq}{\geqslant}

\newcommand{\R}{\mathbb{R}}

\newcommand{\N}{\mathbb{N}}

\newcommand{\cL}{\mathcal{L}}

\newif\iftextstyle
\textstyletrue
\everydisplay\expandafter{\the\everydisplay\textstylefalse}

\DeclarePairedDelimiter{\abs}{\lvert}{\rvert}
\DeclarePairedDelimiter{\norm}{\lVert}{\rVert}

\DeclarePairedDelimiter{\paren}{(}{)}

\DeclarePairedDelimiter{\set}{\{}{\}}

\newcommand{\defeq}{\stackrel{\scriptscriptstyle \text{def}}{=}}

\renewcommand{\figurename}{\scriptsize Fig.}
\allowdisplaybreaks

\newcommand{\sysOmegaD}{\eqref{eqnENSOmega}--\eqref{eqnENSu}\xspace}
\newcommand{\sysWD}{\eqref{eqnW}--\eqref{eqnU}\xspace}
\begin{document}
%\newlength{\ipjwidth}\setlength{\ipjwidth}{339.66878pt}\addtolength{\ipjwidth}{-\linewidth}\typeout{\the\ipjwidth}
\title{Stability of Vortex Solutions to an Extended Navier-Stokes System%
  \thanks{%
    The authors thank the AMS Math Research Communities program (NSF grant DMS 1007980) where this research was initiated, and Center for Nonlinear Analysis (NSF Grants No. DMS-0405343 and DMS-0635983) where part of this research was carried out.
    GG acknowledges partial support from NSF grant DMS 1212141.
    GI acknowledges partial support from NSF grant DMS 1252912, and an Alfred P. Sloan research fellowship.
    JPW thanks the LANL/LDRD program for its support.
  }%
}

\author{%
  Gung-Min Gie\thanks{%
    Department of Mathematics, University of Louisville, Louisville, KY 40292.
    \email{gungmin.gie@louisville.edu}}
  \and
  Christopher Henderson\thanks{%
    Department of Mathematics, Stanford University, Stanford, CA 94305.
    \email{chris@math.stanford.edu}}
  \and
  Gautam Iyer\thanks{%
    Mathematical Sciences, Carnegie Mellon University, Pittsburgh, PA 15213.
    \email{gautam@math.cmu.edu}}
  \and
  Landon Kavlie\thanks{%
    Department of Mathematics, Statistics, and Computer Science, University of Illinois at Chicago, Chicago, IL 60607.
    \email{lkavli2@uic.edu}}
  \and
  Jared P. Whitehead\thanks{%
    Mathematics Department, Brigham Young University, Provo, UT 84602,
    \email{whitehead@mathematics.byu.edu}}
}

\pagestyle{myheadings}
\markboth
  {Stability of Vortex Solutions to an Extended Navier-Stokes System}
  {Gie, Henderson, Iyer, Kavlie and Whitehead}
\maketitle

\maketitle
\begin{abstract}
We study the long-time behavior an extended Navier-Stokes system in $\R^2$ where the incompressibility constraint is relaxed.
This is one of several ``reduced models'' of Grubb and Solonnikov '89 and was revisited recently (Liu, Liu, Pego '07) in bounded domains in order to explain the fast convergence of certain numerical schemes (Johnston, Liu '04).
Our first result shows that if the initial divergence of the fluid velocity is mean zero, then the Oseen vortex is globally asymptotically stable.
This is the same as the Gallay Wayne '05 result for the standard Navier-Stokes equations.
When the initial divergence is not mean zero, we show that the analogue of the Oseen vortex exists and is stable under small perturbations.
For completeness, we also prove global well-posedness of the system we study.
\end{abstract}


%\keywords{Navier-Stokes equation, infinite energy solutions, extended system, long-time behavior, Lyapunov function, asymptotic stability}
%\subjclass{76D05, 35Q30, 76M25, 65M06}

  % !TEX root = oseen.tex

\section{Introduction}\label{sxnIntro}

The dynamics of vortices of the incompressible Navier-Stokes equations play a central role in the study of many problems.
Mathematically, control of the vorticity production~\cites{bblConstantinFefferman93,bblBealeKatoEtAl84} will settle a longstanding open problem regarding global existence of smooth solutions~\cites{bblFefferman06,bblConstantin01b}.
Physically, regions of intense vorticity manifest themselves as cyclones in the atmosphere~\cites{bblMontgomerySmith11,bblDengSmithEtAl12}, and at a slightly decreased intensity as eddies in the oceans~\cites{bblColasMcWilliamsEtAl12,bblPetersenWilliamsEtAl13}.  In all cases, regions of intense vorticity are of vital geophysical (and astrophysical) interest.%and understanding their dynamics is of immense meteorological importance.%
%\sidenote{\vspace{-4\baselineskip}TODO: Cite Andy Majda when we say anything about the weather, or none of you will ever get tenure.}

After many years of intense study (see for instance \cites{bblBen-Artzi94, bblBrezis94, bblCarpio94, bblFujigakiMiyakawa01, bblGallayWayne02,  bblKajikiyaMiyakawa86, bblOliverTiti00,bblMasuda84, bblSchonbek85,bblSchonbek91,bblSchonbek99, bblWiegner87,bblGigaKambe88,bblGigaMiyakawaEtAl88}), the seminal work of Gallay and Wayne~\cite{bblGallayWayne05} proved the existence of a globally stable (infinite energy) vortex in ${\R^2}$, known as the Oseen vortex.
%This built on earlier results of~\cite{} and~\cite{bblCarpio94} which required smallness of some parameter.
Physically, this means that any ${L^1}$ configuration of vortex patches will eventually combine into a ``giant'' vortex and then dissipate like the linear heat equation.
The main result of this paper is the analogue of this result for an extended Navier-Stokes system where the incompressibility constraint is relaxed.

The equations we study are one of several ``reduced models'' of Grubb and Solonnikov~\cites{bblGrubbSolonnikov91,bblGrubbSolonnikov89}.
This model resurfaced recently in~\cite{bblLiuLiuEtAl07} to analyze a stable and efficient numerical scheme proposed in~\cite{bblJohnstonLiu04}.
The numerical scheme is a time discrete, pressure Poisson scheme which improves both stability and efficiency of computation by replacing the incompressibility constraint with an auxiliary equation to determine the pressure.
The formal time continuous limit of this scheme is the system
\begin{equation}\label{eqnExtendedNS1}
  \begin{beqn}[;C]
  \partial_t u + (u\cdot \nabla)u + \nabla p = \Delta u,\\
  \partial_t d  = \Delta d,\\
  d = \nabla \cdot u,
  \end{beqn}
\end{equation}
where $u$ represents the fluid velocity and $p$ the pressure.
We draw attention to the fact that the usual incompressibility constraint, $d = 0$, in the Navier-Stokes equations has been replaced with an evolution equation for $d$.
Of course, if $d = 0$ at time $0$, then it will remain $0$ for all time and the system~\eqref{eqnExtendedNS1} reduces to the standard incompressible Navier-Stokes equations.

In domains with boundary the system~\eqref{eqnExtendedNS1} has been studied by numerous authors~\cites{bblIyerPegoEtAl12, bblJohnstonWangEtAl14, bblLiuLiuEtAl07, bblLiuLiuEtAl09, bblLiuLiuEtAl10, bblIgnatovaIyerEtAl15} both from an analytical and a numerical perspective.
Boundaries, however, cause production of vorticity in a nontrivial manner and make the long time behavior of the vorticity intractable by current methods.
Thus, we study the system~\eqref{eqnExtendedNS1} in ${\R^2}$ where at least the long time behavior of vorticity when $d = 0$ is now reasonably understood~\cite{bblGallayWayne05}.

Since $d$ approaches $0$ asymptotically as $t \to \infty$, we expect that the long time behavior of solutions to~\eqref{eqnExtendedNS1} should be the same as that of the standard incompressible Navier-Stokes equations.
Indeed, our first result (theorem~\ref{thmBetaZero}) shows that this is the case, \emph{provided} the initial divergence~$d_0$ has mean $0$.
In this case, the entropy constructed in~\cite{bblGallayWayne05} can still be used to show global stability of the Oseen vortex.
Surprisingly, if $d_0$ does not have mean $0$, the nonlinearity contributes to the entropy non-trivially and we are unable to show global stability of a steady solution using this method.
Instead when $d_0$ has non-zero mean, we use methods similar to~\cites{bblRodrigues09} and show existence (but not uniqueness) of a solution that is stable under small perturbations globally in time, provided $d_0$ has a small enough mean.
%that in weighted Gaussian spaces we show the 
We are unable to show that this solution is stable under large perturbations.
Further, if $d_0$ has large mean, we are unable to show that this solution is stable even under small perturbations.

\subsection*{Plan of this paper}
In section~\ref{sxnResults} we introduce our notation and state our main results.
Next, in section~\ref{sxnMeanZero} we show that if $\beta = 0$ the Oseen vortex is the global asymptotically stable steady state.
%we adapt the arguments of Gallay and Wayne~\cite{bblGallayWayne05} in order to show that if $\beta = 0$, Oseen's vortex is the steady state.
%While the basic structure of the argument remains unchanged, the addition of the compressible term complicates the analysis.
%Here we are saved by an improved estimate on the decay of the heat equation with mean zero initial data.
Then, in section~\ref{sxnNonMeanZero}, we study the analogue of this result when $\beta \neq 0$.
We find the analogue of the Oseen vortex in this context, but are unable to show a global stability result like in the case when $\beta = 0$.
We instead show that the solution is globally stable under perturbations that are small in Gaussian weighted spaces.
%find a steady solution for the vorticity in the self-similar coordinates, and we adapt the arguments of Rodrigues~\cite{bblRodrigues09} in order to prove stability of this solution.
%The argument involves an energy estimate on the perturbations $W - \alpha W_s$ and $W_s - G$.
The proofs in section~\ref{sxnMeanZero} relied on certain heat kernel like bounds for the vorticity and on relative compactness of complete trajectories.
We prove these in sections~\ref{sxnInequalities} and~\ref{sxnCompactness} respectively.
Finally, to ensure our results long time results are not vacuously true, we conclude this paper with section~\ref{sxnWellPosed}, where briefly discuss global well-posedness for the extended Navier-Stokes system in this context.

  % !TEX root = oseen.tex

\section{Statement of results.}\label{sxnResults}

For our purposes it is more convenient to formulate~\eqref{eqnExtendedNS1} in terms of the vorticity
\begin{equation*}
  \omega \defeq \curl u = \partial_1 u_2 - \partial_2 u_1.
\end{equation*}
Taking the curl of~\eqref{eqnExtendedNS1} gives the system
\begin{gather}
  \label{eqnENSOmega}
    \partial_t \omega + \nabla\cdot( u\omega) = \Delta \omega,
  \\
  \label{eqnENSd}
    \partial_t d = \Delta d,
  \\
  \label{eqnENSu}
    u = \BS \omega + \gradinv d,
\end{gather}
where, $K_{BS}$ and $\gradinv$ are defined by
\begin{equation*}
    K_{BS}(x) \defeq \frac{1}{2\pi} \frac{x^\perp}{|x|^2},
    \qquad\text{and}\qquad
    \gradinv f \defeq \frac{1}{2\pi} \frac{x}{|x|^2} \ast f.
\end{equation*}
Equation~\eqref{eqnENSu} simply recovers $u$ as the unique vector field with divergence $d$ and curl $\omega$.
When $d = 0$, this is simply the Biot-Savart law, hence our notation~$K_{BS}$.

Formally integrating equations~\eqref{eqnENSOmega} and~\eqref{eqnENSd}, one immediately sees that the quantities
\begin{equation}\label{eqnMean}
  \alpha \defeq \int_{\R^2} \omega(x, t) \, dx,
  \qquad\text{and}\qquad
  \beta \defeq \int_{\R^2}  d(x, t) dx
\end{equation}
are constant in time.
The value of $\alpha$ in the long term vortex dynamics is mainly that of a scaling factor and not too important.
The value of $\beta$, however, affects the dynamics (or at least our proofs) dramatically.
We begin by studying the long term vortex dynamics when $\beta = 0$.
In this case we show that the Oseen vortex
defined by
\begin{equation*}
  \tilde \omega(x, t) = \frac{1}{t} G\paren[\Big]{ \frac{x}{\sqrt{t}} }
\end{equation*}
is the globally stable solution,
where 
\begin{equation*}
  G(x) \defeq \frac{1}{4\pi} \exp\paren[\Big]{ \frac{-|x|^2}{4} }
\end{equation*}
is the Gaussian.
%Of course one immediately sees that the pair $(\alpha \tilde \omega, 0)$ is a solution to the system~\sysOmegaD.
We state this as our first result.

\begin{theorem}\label{thmBetaZero}%{{{
Suppose $\omega_0$, $d_0 \in L^1(\R^2)$ are such that $\abs{x} d_0 \in L^1({\R^2})$ and $\beta = 0$.
If the pair $(\omega, d)$ solves the system~\sysOmegaD with initial data $(\omega_0, d_0)$ then for any $p \in [1, \infty]$ we have
\begin{equation}\label{eqnOmegaDDecay}
  \lim_{t\to\infty} 
    t^{1-1/p} \left\| \omega(t, \cdot) - \alpha \tilde \omega(t, \cdot) \right\|_{L^p}
  = 0
  \quad\text{and}\quad
  \sup_{t \geq 0}
    \;
    t^{\frac{3}{2}- \frac{1}{p}} \|d(t, \cdot)\|_{L^p}
    < \infty.
\end{equation}
\end{theorem}%}}}

When $\beta \neq 0$, we are unable to prove a result as strong as theorem~\ref{thmBetaZero}, because a key entropy estimate is destroyed by the nonlinearity.
To formulate our result in this situation, we first identify the analogue of the Oseen vortex.
We show (in section~\ref{sxnWs}) that the radial self-similar solutions to the system~\sysOmegaD are obtained by rescaling $W_s = W_s(\beta)$, where $W_s$ is the unique, radially symmetric, solution of the ODE
\begin{equation}\label{eqnWs}
  \frac{\partial_r W_s}{W_s} = \frac{-r}{2}
    + \frac{\beta}{2\pi r} \paren[\big]{1 - e^{-r^2/4}},
    \quad
    \text{with normalization }
    \int_{\R^2} W_s \, dx = 1.
\end{equation}
A direct calculation shows that the pair $(\alpha \tilde \omega_\beta, \tilde d_\beta)$ defined by
\begin{equation}\label{eqnWsInXandT}
  \tilde \omega_\beta( x, t ) = \frac{1}{t} W_s\paren[\Big]{\frac{x}{\sqrt{t}} },
  \qquad
  \tilde d_\beta( x, t ) = \frac{\beta}{t} G\paren[\Big]{\frac{x}{\sqrt{t}} },
\end{equation}
is a radially symmetric self-similar solution to the system~\sysOmegaD, making $\tilde \omega_\beta$ the analogue of the Oseen vortex.
When $\beta = 0$, we see $W_s$ is exactly the Gaussian $G$, but this is no longer true when $\beta \neq 0$.
When $\beta < 4\pi$ the shape of $W_s$ is similar to that of the Gaussian in that $W_s$ attains it's maximum at $0$ and is strictly decreasing for $r > 0$.
When $\beta > 4 \pi$, however, $W_s$ attains its maximum at some $r_0 > 0$ and the profile looks like that of a ``vortex ring'' (see figure~\ref{fgrWs}).
For any $\beta \neq 0$, the interaction between $W_s$ and the nonlinearity is largely responsible for the failure in our proof of theorem~\ref{thmBetaZero}.
\begin{figure}[htb]
  \centering
  \includegraphics[width=.8\linewidth]{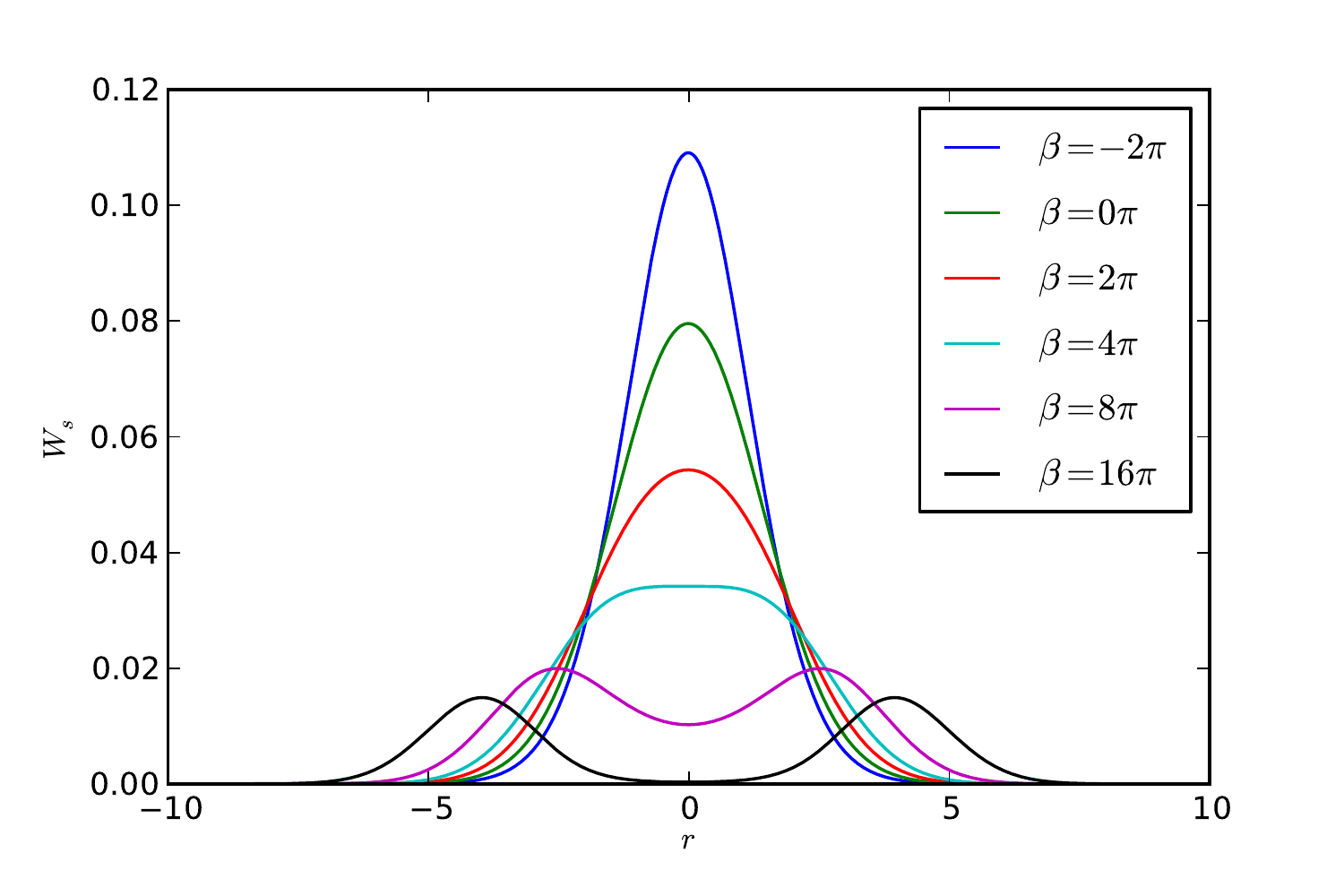}
  \caption{Plots of $W_s$ vs $r$ for $\beta \in \set{-2\pi, 0, 2\pi, 4\pi, 8\pi, 16\pi}$.}\label{fgrWs}
\end{figure}

Our main result when $\beta \neq 0$ uses the Gaussian weighted spaces appearing in~\cites{bblGallayWayne02, bblGallayWayne05, bblRodrigues09} and shows that the solution $(\alpha \tilde \omega_\beta, \tilde d_\beta )$ is stable under small perturbations.
Explicitly, define the weighted spaces $L^2_w$ by
\begin{equation}\label{eqnWeightedSpaces}
  % L^2(m) \defeq \{f \in L^2(\R^2): \|f\|_{L^2(m)} < \infty\}
  % \quad\text{and}\quad
  L^2_w \defeq \{f \in L^2(\R^2): \|f\|_w < \infty\},
  \quad\text{where}\quad
  \|f\|_w^2 \defeq \int G(x)^{-1} |f(x)|^2 dx.
\end{equation}
Now our stability result when $\beta \neq 0$ is as follows:

\begin{theorem}\label{thmBetaNonZero}
Let $t_0 > 0$ and $(\omega, d)$ solve the system~\sysOmegaD on the time interval $[t_0, \infty)$.
%$\omega(t_0) = \omega_0$ and $d(t_0) = d_0 \in L^2_w$.
For any $\gamma \in (0, 1/2)$, there exists $\epsilon_0 = \epsilon_0( \gamma) > 0$ such that if
\[
  |\beta|(1+|\alpha|)
  +
  %\|\omega_0 - \alpha W_s\|_w
  %\|d_0 - \beta G\|_w
  \norm[\big]{ \omega(t_0) - \alpha \tilde \omega_\beta(t_0) }_w
  +
  \norm[\big]{ d(t_0) - \tilde d_\beta(t_0) }_w
  \leq \epsilon_0,
\]
then
%\[\begin{split}
%	\lim_{t\to\infty} t^{\gamma}
%		&\left\|G^{-1/2}(\cdot/\sqrt{t}) \left(\omega(x, t) - \frac{\alpha}{t} W_s(\cdot/\sqrt{t}) \right)\right\|_w\\
%		+ t^{\gamma} &\left\|G^{-1/2}(\cdot/\sqrt{t}) \left(d(x, t) - \frac{\beta}{t} G(\cdot/\sqrt{t}) \right)\right\|_w
%	= 0
%.\end{split}\]
\begin{equation*}
  \lim_{t \to \infty}
    t^\gamma
    \norm{ \tilde G(t)^{-1/2} \paren{\omega(t) - \alpha \tilde \omega_\beta(t) } }_w
  = 0
\end{equation*}
and
\begin{equation*}
\sup_{t\geq 0}\:
    t^{1/2}
    \norm{ \tilde G(t)^{-1/2} \paren{d(t) - \tilde d_\beta(t) } }_w <\infty.
\end{equation*}
Here $\tilde G(x, t) = G( x / \sqrt{t} )$ is the rescaled Gaussian.
\end{theorem}

When $\beta = 0$, the function $\tilde \omega_\beta = \tilde \omega$, and theorem~\ref{thmBetaZero} proves stability of $\tilde \omega$ (albeit under a different norm) without any smallness assumption on the perturbation.

Finally, to ensure that theorems~\ref{thmBetaZero} and~\ref{thmBetaNonZero} are not vacuously true, we establish global existence of solutions to the system~\sysOmegaD.
While a little work has been done on this system in $\R^2$, the existence and uniqueness theory is not altogether far from the classical theory, and we address this next.

\begin{proposition}\label{ppnExistence}
Define the space $X$ to be either $L^1$ or $L^2_w$.
If $\omega_0,d_0 \in X$, then there exists a unique time global solution $(\omega, d)$ to the system~\sysOmegaD in $X$ with initial data $(\omega_0, d_0)$.
\end{proposition}

The proof of this proposition follows a similar structure to results in~\cites{bblBen-Artzi94,bblBrezis94,bblGallayWayne02,bblGallayWayne05,bblKato94,bblRodrigues09}, and we do not provide a complete proof.
However, for convenience of the reader, we sketch a brief outline in section~\ref{sxnWellPosed}.%  We will, however, provide a brief discussion at various points in the article.  To be specific, we discuss extensions of classical well-posedness results in section~\ref{sxnWellPosed}, we discuss well-posedness in $L^2(m)$ in section~\ref{sxnMeanZero}, and we discuss well-posedness in $L^2_w$ in section~\ref{sxnNonMeanZero}.

\section{Global stability for mean zero initial divergence.}\label{sxnMeanZero}

We devote this section to proving theorem~\ref{thmBetaZero}.
The main idea in the case where $\beta = 0$ is the same as that used by Gallay and Wayne in~\cite{bblGallayWayne05}.
However, to use this method, certain compactness criteria and vorticity bounds need to be established.
In order to present a self contained treatment, we begin with the heart of the matter (following~\cite{bblGallayWayne05}), and only state the compactness criteria where required.
We postpone the proofs of the vorticity bounds and these criteria to sections~\ref{sxnInequalities} and~\ref{sxnCompactness} respectively.

\subsection{Reformulation using self-similar coordinates.}\label{sxnSelfSimilar}

We begin by reformulating theorem~\ref{thmBetaZero} in the natural self-similar coordinates associated to~\eqref{eqnExtendedNS1}.

\begin{proof}[Proof of theorem~\ref{thmBetaZero}]%{{{
Define the coordinates $\xi$ and $\tau$ by
\begin{equation}
	\xi \defeq \frac{x}{\sqrt{t}},
	\quad
	\tau \defeq \log(t),
\end{equation}
and the rescaled velocity, vorticity, and divergence by
\begin{equation}\label{eqnVarChange}
    %u( x, t ) = \frac{1}{\sqrt{t}} U( \xi, \tau ),
     U( \xi, \tau ) \defeq \sqrt{t} u( x, t ),
    \quad
    %\omega(x,t) = \frac{1}{t} W(\xi, \tau)
    W(\xi, \tau) \defeq t \omega(x,t),
    \quad\text{ and }\quad
    %d(x,t) = \frac{1}{t} D(\xi, \tau).
    D(\xi, \tau) \defeq t d(x,t).
\end{equation}
With this transformation the system~\sysOmegaD becomes
\begin{gather}
  \label{eqnW}
    \partial_\tau W + \nabla\cdot (UW) = \cL W,
  \\
  \label{eqnD}
    \partial_\tau D = \cL D,
  \\
  \label{eqnU}
    U = \BS W + \gradinv D.
\end{gather}
where $\mathcal{L}$ is the operator defined by
\begin{equation}\label{eqnCLDef}
  \cL f \defeq \lap f + \frac{1}{2} \xi \cdot \nabla f + f.
\end{equation}

In the rescaled variables we will prove the following result:
\begin{proposition}\label{ppnMeanZero}
%Let $(\omega, d)$ solve the system~\sysOmegaD with $\omega_0$, $(1 + \abs{x})d_0 \in L^1(\mathbb{R}^2)$.
Let $(W, D)$ solve the system~\sysWD with initial data $(W_0, D_0)$ such that $W_0, (1 + \abs{\xi})D_0 \in L^1(\mathbb{R}^2)$.
If $\alpha = \int W_0 \, d\xi$ and $\beta = \int D_0 \, d\xi = 0$, then
\begin{equation}\label{eqnMeanZeroConv}
	\lim_{\tau\to\infty} \left\| W - \alpha G\right\|_{L^p} = 0
	\quad\text{and}\quad
	\sup_{\tau\geq 0} \; e^{\frac{\tau}{2}}
	  \left\| D \right\|_{L^p} < \infty
\end{equation}
for any $p \in [1, \infty]$.
\end{proposition}
Undoing the change of variables immediately yields theorem~\ref{thmBetaZero}.
\end{proof}%}}}

Before proving proposition~\ref{ppnMeanZero} we pause momentarily to explain why the proof in this case is similar to the proof in~\cite{bblGallayWayne05} for the standard Navier-Stokes equations.
The only mean zero function $D$ that decays sufficiently at infinity and is an equilibrium solution to~\eqref{eqnD} is the $0$ function, in which case the system~\sysWD reduces to the standard Navier-Stokes equations in self-similar coordinates.
Thus, when $\beta = 0$, the long time dynamics of the system~\sysWD should be similar to that of the standard Navier-Stokes equations (in self-similar coordinates).
Indeed, as we show below, the key step of the proof in~\cite{bblGallayWayne05} goes through almost unchanged.
Of course, the required bounds and compactness estimates leading up to this still require work to prove and, for clarity of presentation, we postpone their proofs to sections~\ref{sxnInequalities} and~\ref{sxnCompactness}.

The proof of proposition~\ref{ppnMeanZero} consists of two main steps.
The first step is to establish relative compactness of trajectories to the system~\sysWD in the space $L^1$ and is our next lemma.

\begin{lemma}\label{lmaCompactness}%{{{
%\sidenote{Tue 10/28 GI: Should perhaps clarify that $C^0$ doesn't mean vanish at $\infty$. Maybe write it as $C(...)$ instead?}
Suppose that $W$ and $D$ solve the system~\sysWD in $C^0([0,\infty), L^1(\R^2)\times L^1(\R^2))$.
Then the trajectory $\{(W(\tau),D(\tau))\}_{\tau\in[0,\infty)}$ is relatively compact in $L^1(\R^2)$.
Further,
\begin{equation}\label{eqnWbound}
  \abs{W(\xi, \tau)} \leq
    C \int_{\R^2}
	\exp \paren[\Big]{ \frac{-\abs{\xi - \eta e^{-\tau/2}}^2 }{C} }
	\abs{W_0(\eta)}
	\, d\eta.
\end{equation}
for some constant $C$ which depends only on $\norm{W_0}_{L^1}$ and $\norm{(1 + \abs{\xi})D_0(\xi)}_{L^1}$.
\end{lemma}%}}}

The second step in the proof of proposition~\ref{ppnMeanZero} is to characterize complete trajectories of the system~\sysWD.
%show that the only complete trajectory of the system~\sysWD which is bounded in a certain weighted $L^2$ space corresponds to $W = \alpha G$ and $D = 0$.
%We recall that $\alpha = \int_{\R^2} W \, d\xi$ and by assumption $\int_{\R^2} D \, d\xi = 0$.
To do this we need to introduce a weighted $L^p$ space.
For any $m \geq 0$, $p \in [1, \infty)$ we define the space $L^p(m)$ by
\begin{equation*}
	L^p(m) = \set[\Big]{
	  f \in L^p: \|f\|_{L^p(m)} < \infty,
	  \text{ where }
	  \|f\|_{L^p(m)}^p = \int (1 + |\xi|^2)^\frac{pm}{2} |f(\xi)|^p d\xi
	}.
\end{equation*}
It turns out that the only complete trajectories of the system~\sysWD that are bounded in $L^2(m)$ are scalar multiples of the Gaussian.
This is our next lemma.

\begin{lemma}\label{lmaMeanZeroCTraj}%Warning: Used to be a proposition, now downgraded to a lemma.
Let $m > 3$ and suppose that $\{(W(\tau),D(\tau))\}$ is a complete trajectory of the system~\sysWD which is bounded in $L^2(m)$.
Then, if $\int W_0 = \alpha$ and $\int D_0 = 0$ we must have $W(\tau) = \alpha G$ and $D = 0$ for all $\tau$.
%Here $\alpha = \int W_0 d\xi$.
\end{lemma}

Momentarily postponing the proofs of lemmas~\ref{lmaCompactness} and~\ref{lmaMeanZeroCTraj}, we prove proposition~\ref{ppnMeanZero}.

\begin{proof}[Proof of proposition~\ref{ppnMeanZero}]%{{{
%\textbf{This is the statement, should be moved elsewhere}  Let $(W,D) \in C^0([0,\infty), L^1(\R^2))$ be a solution to the system~\sysWD, with $W_0\in L^1$ and $D_0 \in L^1(1)$.  Then $W$ tends to $\alpha G$ in $L^1$, where
%\[\alpha = \int W_0(\xi) d\xi = \int \omega_0(y)dy.\]
%
%SOMETHING SOMETHING SOMETHING D GOES TO ZERO.
Let $\Omega$ be the $\omega$-limit set of the trajectory $(W, D)$.
Since lemma~\ref{lmaCompactness} guarantees $\{W(\tau)\}$ and $\{D(\tau)\}$ are relatively compact in $L^1$, $\Omega$ must be non-empty, compact, and fully invariant under the evolution of the system~\sysWD.
Consequently, the trajectory of any $(\overline{W}, \overline{D}) \in \Omega$ must be complete.

Further, the upper bound~\eqref{eqnWbound} implies $\overline{W}$ is bounded above by a Gaussian.
To see this, choose a sequence of times $\tau_n \to \infty$ such that $(W(\tau_n)) \to \overline{W}$ in $L^1$ and almost everywhere.
Now dominated convergence and~\eqref{eqnWbound} imply
\begin{equation*}
  \begin{split}
    |\overline{W}(\xi)|
      &= \lim_{n\to\infty} |W(\xi, \tau_n)|
      \leq \lim_{n\to\infty}
	  C \int
	      \exp\paren[\Big]{ \frac{-|\xi-\eta e^{\frac{-\tau_n}{2}}|^2}{C} } |W_0(\eta)|
	      \, d\eta
    \\
    &\leq
      C \|W_0\|_{L^1} \exp\paren[\Big]{ \frac{-|\xi|^2}{C} }.
  \end{split}
\end{equation*}
Consequently $\Omega \subset L^2(m)^2$ for every $m$.

This implies that for any $(\overline{W}, \overline{D})$, the associated complete trajectory is bounded in $L^2(m)^2$ for every $m$.
Thus lemma~\ref{lmaMeanZeroCTraj} shows $\Omega \subset \{(\theta G,0): \theta \in \R\}$.
Since total mass is invariant under the flow (and $\Omega \neq \emptyset$), it follows that $\Omega = \{(\alpha G, 0)\}$, where $\alpha$ is defined in~\eqref{eqnMean}.
Since $\Omega$ contains exactly one element and $(W(\tau), D(\tau))$ is relatively compact in $L^1$, this immediately implies the first equality in~\eqref{eqnMeanZeroConv} for $p=1$.
Combined with the Gaussian upper bound implied by~\eqref{eqnWbound}, we obtain the first inequality in~\eqref{eqnMeanZeroConv} for any $p < \infty$.

The proof for $p = \infty$ uses bounds on the semigroup generated by the operator $\mathcal{L}$ and an integral representation for $W$.
Since we develop these bounds in section~\ref{sxnCompactness}, we prove $L^\infty$ convergence as lemma~\ref{lmaLinf} at the end of section~\ref{sxnCompactness}.

%{\color{blue} Convergence in $L^\infty$ follows from convergence in $L^p$ and a semigroup argument.  NOT SURE WHERE TO PUT THIS.} For $p = \infty$, we claim the inequality
%\begin{equation}\label{eqnWL1toLinf}
%  \norm{W(\tau) - \alpha G}_{L^\infty} \leq C \norm{W(\tau/2) - \alpha G}_{L^1}.
%\end{equation}
%\sidenote{
%\textcolor{red}{GG (11/02)\\
%Will need to add lemma TODO in section~\ref{sxnInequalities}?
%}
%}
%This is the content of lemma~\highlight{TODO (which we prove in section TODO)} and follows from a heat kernel estimate which is similar to~\eqref{eqnWbound}.
%Inequality~\eqref{eqnWL1toLinf} immediately implies $W(\tau) \to \alpha G$ in $L^\infty$ as $\tau \to \infty$, and proves the first inequality in~\eqref{eqnMeanZeroConv} for $p = \infty$.

The second inequality in~\eqref{eqnMeanZeroConv} follows directly from the explicit solution formula for the heat equation.
Since this will also be used later, we extract it as a lemma.
\begin{lemma}\label{d_estimates}%{{{
%\sidenote{Fri 11/14 GI: Powers were improved. Check references.}
Let
%that
$D$ be a solution to~\eqref{eqnD} with initial data $D_0$.
Suppose
\begin{equation*}
  \int_{\R^2} (1 + \abs{\xi} ) \abs{D_0(\xi)} \, d\xi < \infty
  \quad\text{and}\quad
  \int_{\R^2} D_0 \, d\xi = 0.
\end{equation*}
There there exists a universal constant $C>0$ such that
\begin{equation}\label{eqnDfastDecay}
  %\add{\norm{\gradinv D_p}_{L^\infty}} +
  \norm{D(\tau)}_{L^p} \leq C e^{-\tau/2} \int_{\R^2} (1 + \abs{\xi}) \abs{D_0(\xi)} \, d\xi 
\end{equation}
for all $p \in [1,\infty]$.
\end{lemma}%}}}

We remark that the decay rate of $D$ to $0$ being faster than that of the rescaled heat kernel is because the initial data has mean-zero.
This concludes the proof of proposition~\ref{ppnMeanZero}.
\end{proof}%}}}

It remains to prove lemmas~\ref{lmaCompactness}--\ref{d_estimates}.
The proof of lemma~\ref{d_estimates} is short, and we present it here.

\begin{proof}[Proof of lemma~\ref{d_estimates}]%{{{
Since the heat kernel in $x$-$t$ coordinates is common knowledge, we return to the $x$-$t$ coordinates and prove $d$ satisfies the second inequality in~\eqref{eqnOmegaDDecay}.
Let $\bar G(x, t) = G(x / \sqrt{t}) / t$ be the heat-kernel.
Observe
\begin{align*}
  \norm{d(t)}_{L^p}
    &= \norm{d_0 * \bar G(t)}_{L^p}
    = \norm[\Big]{
	  \int_{\R^2} d_0(y) \bar G(x-y, t) \, dy
	}_{L^p(x)}
  \\
    &= \norm[\Big]{
	  \int_{\R^2} d_0(y) \paren[\big]{ \bar G(x-y, t) - \bar G(x, t) } \, dy
	}_{L^p(x)}
    \\
    &\leq 
      \frac{1}{t^{1 - 1/p}}
      \int_{\R^2} \abs{d_0(y)}
	\norm[\big]{ G( x - t^{-1/2} y ) - G(x) }_{L^p(x)} \, dy
    \\
    & \leq
      \frac{C}{t^{\frac{3}{2} - \frac{1}{p}}} \int_{\R^2} \abs{y d_0(y)} \, dy,
\end{align*}
which implies the second inequality in~\eqref{eqnOmegaDDecay} and concludes the proof.
\end{proof}%}}}

The proof of lemma~\ref{lmaCompactness} is technical; we postpone the proof of~\eqref{eqnWbound} to section~\ref{sxnInequalities} and the proof of compactness to section~\ref{sxnCompactness}. We prove lemma~\ref{lmaMeanZeroCTraj} in section~\ref{sxnLyapunov}.

\subsection{Characterization of complete trajectories.}\label{sxnLyapunov}%{{{1

%The Lyapunov functions are similar to those used by Gallay and Wayne.  

The characterization of complete trajectories to the system~\sysWD when $\beta = 0$ is identical to the characterization of complete trajectories of the 2D Navier-Stokes equations presented in~\cite{bblGallayWayne05}.
Since the proof is short and elegant, we reproduce it here for the reader's convenience.

There are two steps to this proof:
  First, show that in a complete trajectory both $W$ and $D$ must have constant sign.
  Of course, since $D$ is mean-zero, this forces $D = 0$ identically, and reduces to the situation already considered by Gallay and Wayne~\cite{bblGallayWayne05}.
  Second, the most interesting step, is to use the Boltzmann entropy functional to show that $W$ must be a scalar multiple of a Gaussian.
  This is exactly what fails in the case where $D$ is not mean zero.

We state each of these steps as lemmas, below:

\begin{lemma}\label{lmaSign}%{{{
  Suppose $m>3$ and $(W,D)\in C^0(\R, L^2(m)^2)$ is a solution of the system~\sysWD which is bounded in $L^2(m)$.
  Then both $W$ and $D$ must have constant sign.
\end{lemma}%}}}

\begin{lemma}\label{lmaEntropy}%{{{
  Let $(W, D)$ be a solution to the system~\sysWD with $W_0 \in L^2(m)$, $D_0 = 0$, $W_0 \geq 0$.
  For the relative entropy $H$ given by 
  \begin{equation}\label{eqnHdef}
    H(W) = \int_{\R^2} W \ln \paren[\Big]{\frac{W}{G}} \, d\xi,
  \end{equation}
  we have
  \begin{equation}\label{eqnHdecay}
    \partial_\tau H
      = -\int_{\R^2}
	    W
	    \abs[\Big]{
	      \nabla \ln \paren[\Big]{ \frac{W}{G} }
	    }^2
	    \, d\xi.
  \end{equation}
\end{lemma}%}}}

Lemma~\ref{lmaMeanZeroCTraj} immediately follows from lemmas~\ref{lmaSign}--\ref{lmaEntropy}, and we spell it out here for completeness.

\begin{proof}[Proof of lemma~\ref{lmaMeanZeroCTraj}]%{{{
  By lemma~\ref{lmaSign}, we know that both $W$ and $D$ have constant sign.
  Since $\int D = 0$, this forces $D = 0$ identically.
  Further, by symmetry we can assume $W \geq 0$.

  Note that by the comparison principle the set $L^2(m) \cap \set{ \tilde{W} \geq 0}$ is invariant under the dynamics of the system~\sysWD.
  Restricting our attention to this set, we observe that the entropy $H$ is strictly decreasing except on the set of equilibria $\tilde{W} = \theta G$.
  By LaSalle's invariance principle this implies that $W = \theta G$ for some $\theta$.
  Since $\int W = \alpha$ this forces $\theta = \alpha$ concluding the proof.
\end{proof}%}}}

It remains to prove lemmas~\ref{lmaSign} and~\ref{lmaEntropy}, which we do in sections~\ref{sxnSign} and~\ref{sxnEntropy} respectively.

\subsubsection{The sign of complete trajectories.}\label{sxnSign}

The main idea behind the proof of lemma~\ref{lmaSign} is that the $L^1$ norm can be used as a Lyapunov functional.
However, we first need a relative compactness lemma to guarantee that the $\alpha$ and $\omega$-limit sets are non-empty, and we state this next.

\begin{lemma}\label{lmaWCompact}%{{{
  Let $m>3$ and suppose $(W,D)\in C^0(\R, L^2(m)^2)$ is a solution to the system~\sysWD which is bounded in $L^2(m)$.
   The trajectory $\{( W(\tau), D(\tau) )\}_{\tau\in \R}$ is relatively compact in $L^2(m)$.
\end{lemma}%}}}

Lemma~\ref{lmaWCompact} is also used in the proof of lemma~\ref{lmaCompactness}, and we defer its proof to section~\ref{sxnCompactness}.
We prove lemma~\ref{lmaSign} next.

\begin{proof}[Proof of lemma~\ref{lmaSign}]%{{{
Define the Lyapunov function $\Phi$ by $\Phi(W, D) = \norm{W}_{L^1} + \norm{D}_{L^1}$.
%\begin{equation*}
%  \Phi(W,D) \defeq \int_{\R^2} (|W| + |D|) \, d\xi.
%\end{equation*}
We claim that $\Phi$ is always decreasing, and is strictly decreasing in time if and only if one of $W$ and $D$ does not have a constant sign.
To see this, define $W^+$ and $W^-$ to be the solutions to
\begin{equation*}
  \partial_\tau W^{+} + \dv (U W^+) = \mathcal{L} W^+
  \quad\text{and}\quad
  \partial_\tau W^{-} + \dv (U W^-) = \mathcal{L} W^-,
\end{equation*}
with initial data $W_0^+ = \max\set{W, 0}$ and $W_0^- = \max\set{-W,0}$ respectively.
We clarify that $U = \BS W$ here, and does not depend on $W^+$ or $W^-$.
Clearly $W^\pm \geq 0$ and $W = W^+ - W^-$ for all time.
Further, if both $W^+$ and $W^-$ are non-zero initially, the strong maximum principle implies that for any $\tau > 0$ the supports of $W^\pm(\tau)$ will necessarily intersect.
Consequently, for any $\tau > 0$,
\begin{multline}
  \int_{\R^2} \abs{W(\xi, \tau)} \, d\xi
    < \int_{\R^2} \paren{W^+(\xi, \tau) + W^-(\xi, \tau)} \, d\xi
    \\
    = \int_{\R^2} \paren{W_0^+(\xi) + W_0^-(\xi)} \, d\xi
    = \int_{\R^2} \abs{W_0} \, d\xi.\label{L1_decrease}
\end{multline}
A similar argument can be applied to $D$ and replacing $\tau = 0$ with any arbitrary time $\tau_0$ will show that $\Phi$ is strictly decreasing in time if and only if either $W$ or $D$ do not have a constant sign.

To see that complete trajectories have constant sign, we appeal to lemma~\ref{lmaWCompact} to guarantee that the trajectory $\set{(W(\tau), D(\tau)}_{\tau \in \R}$ has both an $\alpha$ and an $\omega$-limit.
Now choose two sequences of times $(\overline{\tau}_n) \to \infty$ and $(\underline{\tau}_n) \to -\infty$ such that
\begin{equation*}
  \underline{W} = \lim W(\underline{\tau}_n)
  \quad\text{and}\quad
  \overline{W} = \lim W( \overline{\tau}_n )
  \quad \text{in $L^2(m)$}.
\end{equation*}
Since $\int W$ is conserved we must have $\int \overline{W} = \int \underline{W}$.
Further, by LaSalle's invariance principle both $\overline{W}$ and $\underline{W}$ have constant sign.
Consequently, for any $\tau \in \R$,
\begin{equation*}
  \abs[\Big]{ \int_{\R^2} \underline{W} \, d\xi }
  = \int_{\R^2} \abs{ \underline{W} } \, d\xi
  \geq \int_{\R^2} \abs{ W(\tau) } \, d\xi
  \geq \int_{\R^2} |\overline{ W }| \, d\xi
  = \abs[\Big]{ \int_{\R^2} \overline{ W } \, d\xi }
  = \abs[\Big]{ \int_{\R^2} \underline{ W } \, d\xi }.
\end{equation*}
Hence, $\int |W(\tau)|d\xi$ is constant in $\tau$.  This, along with~\eqref{L1_decrease}, shows that $W$ has a constant sign.
A similar argument can be applied to $D$.
This shows that $\Phi$ is constant in time and hence both $W$ and $D$ must have constant sign.
\end{proof}%}}}

\subsubsection{Decay of the Boltzmann entropy.}\label{sxnEntropy}

The use of the relative entropy $H$ in this context was suggested
%to the authors of~\cite{bblGallayWayne05}
by C. Villani, and the decay (when $D = 0$) is a direct calculation that was carried out in~\cite{bblGallayWayne05}*{lemma 3.2}.
We briefly sketch a few details here for the readers convenience.

\begin{proof}[Proof of lemma~\ref{lmaEntropy}]%{{{
  Differentiating~\eqref{eqnHdef} with respect to $\tau$ gives
  \begin{equation*}
    \partial_\tau H
      = \int_{\R^2} \paren[\Big]{1 + \ln \paren[\Big]{\frac{W}{G}} } \partial_\tau W
      = \int_{\R^2} \paren[\Big]{1 + \ln \paren[\Big]{\frac{W}{G}} }
	\paren[\big]{\mathcal{L} W - \dv (U W)}.
  \end{equation*}
  Using the identity $\nabla G / G = -\xi / 2$ and the term involving $\mathcal{L}$ simplifies to
  \begin{multline*}
    \int_{\R^2}
      \paren[\Big]{1 + \ln\paren[\Big]{\frac{W}{G}} } \mathcal{L} W
      \, d\xi 
      = -\int_{\R^2}
	  \paren[\Big]{\nabla W + \frac{\xi}{2} W}
	  \cdot
	  \paren[\Big]{\frac{\nabla W}{W} - \frac{\nabla G}{G} }
	  \, d\xi 
	\\
      = -\int_{\R^2} W
	  \abs[\Big]{ \frac{\nabla W}{W} - \frac{\nabla G}{G} }^2
	  \, d\xi 
      = -\int_{\R^2} W
	    \abs[\Big]{
	      \nabla \ln \paren[\Big]{ \frac{W}{G} }
	    }^2
	    \, d\xi.
  \end{multline*}

  We claim the convection terms integrate to $0$.
  Indeed,
  \begin{equation*}
    - \int_{\R^2}
      \paren[\Big]{1 + \ln\paren[\Big]{\frac{W}{G}} } \dv (U W) \, d\xi 
      =
      \int_{\R^2} U \cdot \nabla W \, d\xi + \frac{1}{2} \int_{\R^2} W U \cdot \xi \, \, d\xi.
  \end{equation*}
  The first term on the right clearly integrates to $0$.
  If $U$ decayed sufficiently at infinity, we can write $W = \curl U$, integrate the second term by parts, and obtain
  \begin{equation}\label{eqnNonMagic}
    \frac{1}{2} \int_{\R^2} W U \cdot \xi \, \, d\xi
      = \frac{1}{4} \int_{\R^2} \xi \cdot \gradperp \abs{U}^2 \, d\xi 
      = 0.
  \end{equation}
  Without the decay assumption one can use the Biot-Savart law and Fubini's theorem (see for instance~\cite{bblGallayWayne05}*{lemma 3.2}) and still show this term integrates to $0$.
  This immediately yields~\eqref{eqnHdecay} as desired.
\end{proof}%}}}

% Thu 10/30 GI: Cut out old stuff and put it in cut_meanzero.tex
%\input{cut_meanzero}

  % !TEX root = oseen.tex

\section{Stability when the initial divergence has non-zero mean}\label{sxnNonMeanZero}

In this section, we study the long time behaviour of the system~\sysOmegaD when $\beta \neq 0$ (i.e.\ when the mean of the initial divergence is non-zero) and prove theorem~\ref{thmBetaNonZero}.
Unlike the behaviour in section~\ref{sxnMeanZero}, the divergence $D$ of the equilibrium solution to the system~\sysWD is non-zero.
Consequently, the steady state of the system~\sysWD is no longer a Gaussian (like the Oseen vortex), but the radial function $W_s$ defined by~\eqref{eqnWs}.
We remark, however, that different, non-radial, steady solutions to the system~\sysWD may exist and we can neither prove nor disprove their existence.
%A radial steady state, denoted by $W_s$, can be computed explicitly using equation~\eqref{eqnWs}, however we can not prove or disprove the existence of non-radial steady solutions to the system~\sysWD, even for small $\beta$.

Further it turns out that the radial state $W_s$ doesn't ``play nice'' with the non-linearity.
We are unable to show decay of the analogue of the Boltzmann entropy~\eqref{eqnHdef}, which is a key step in both~\cite{bblGallayWayne05} and the proof of theorem~\ref{thmBetaZero}.
We can, however, show that $W_s$ is stable under small perturbations globally in time (theorem~\ref{thmBetaNonZero}) using techniques that are similar to those in~\cites{bblRodrigues09,bblGallayMaekawa13}.
This is the main goal of this section.

In section~\ref{sxnWs}, we derive an explicit equation for the radial steady state~$W_s$.
In section~\ref{sxnEntropyFail}, we compute the evolution of the Boltzmann entropy functional mainly to point out the breaking point of the argument of Gallay and Wayne~\cite{bblGallayWayne05}.
In section~\ref{sxnRodrigues}, we use a different method (similar to that in~\cite{bblRodrigues09}) to prove stability under small perturbations (theorem~\ref{thmBetaNonZero}) modulo the proofs of a few estimates which are presented in section~\ref{sxnDivBounds}.

\subsection{The radial steady state}\label{sxnWs}%{{{1

Since the equation for $D$ is linear, we find that $D \to \beta G$ as $\tau \to \infty$.  This can be seen, for instance, by noticing that $D - \beta G$ satisfies the heat equation in Euclidean coordinates with initial mean zero.  An argument analogous to the proof of lemma~\ref{d_estimates} gives the precise decay.
%This can be seen, for instance, using a Lyapunov function (as in section~\ref{sxnLyapunov}) or by a direct argument using the inequalities we prove later (section~\ref{sxnInequalities}).
Turning to $W$, we denote the steady state by $W_s$.
For convenience, we normalize $W_s$ so that $\int W_s d\xi = 1$.
We claim that a unique radial steady state exists, and is exactly given by~\eqref{eqnWs}.
(We can not, however, rule out the possibility that other non-radial steady states exist.)

To see that the unique radial steady state satisfies~\eqref{eqnWs}, we use equation~\eqref{eqnW} to obtain
\[
	0 = -\left(\BS W_s\right) \cdot \nabla W_s - \beta\nabla\cdot\left((\gradinv G)  W_s\right)+ \cL W_s,
\]
in $L^2_w$.
Under the assumption that $W_s$ is radial, $\BS W_s \cdot \nabla W_s = 0$ and hence
\[
  \beta\nabla\cdot\left((\gradinv G)  W_s\right) = \cL W_s = \nabla \cdot \paren[\Big]{G \nabla\frac{W_s}{G} }.
\]
Consequently,
\[
	\gradperp \varphi = -\beta \gradinv G W_s + G \nabla \frac{W_s}{G}.
\]
for some function $\varphi$.
Since the right hand side is radially pointing and smooth, we must have $\gradperp \varphi = 0$ identically.

Switching to polar coordinates immediately shows that $W_s$ satisfies~\eqref{eqnWs}, and reverting back to the $x$ and $t$ coordinates shows that $(\tilde \omega_\beta, \tilde d_\beta)$, defined in~\eqref{eqnWsInXandT}, is the unique radially symmetric, self-similar solution to the system~\sysOmegaD.

\subsection{The Boltzmann entropy.}\label{sxnEntropyFail}%{{{1

Before embarking on the proof of theorem~\ref{thmBetaNonZero}, we briefly study the analogue of the Boltzmann entropy in this situation.
Naturally, the Gaussian in this context needs to be replaced with $W_s$, the solution to~\eqref{eqnWs}, and so~\eqref{eqnHdef} now becomes
\begin{equation*}
  H(W) = \int_{\R^2} W \ln \paren[\Big]{\frac{W}{W_s}} \, d\xi.
\end{equation*}
Computing $\partial_\tau H$ and performing a calculation similar to that in section~\ref{sxnEntropy} we obtain
\begin{equation*}
  \begin{split}
  \partial_\tau H &=
    \int_{\R^2}
      W (\BS W) \cdot
      \paren[\Big]{
	\frac{\nabla W}{W} - \frac{\nabla W_s}{W_s}
      }
      \, d\xi
    -\int_{\R^2} W
	\abs[\Big]{ \frac{ \nabla W}{W} - \frac{\nabla W_s}{W_s} }^2
       \, d\xi
    \\
    &=
    - \int_{\R^2} W (\BS W) \cdot \frac{\nabla W_s}{W_s} \, d\xi
    -\int_{\R^2} W
	\abs[\Big]{ \frac{ \nabla W}{W} - \frac{\nabla W_s}{W_s} }^2
       \, d\xi.
  \end{split}
\end{equation*}
The second term is of course always negative.
The first term can be simplified using~\eqref{eqnWs} to
\begin{multline*}
  - \int_{\R^2} W (\BS W) \cdot \frac{\nabla W_s}{W_s} \, d\xi
  \\
  = \int_{\R^2} W (\BS W) \cdot \frac{\xi}{2} \, d\xi 
    + \beta \int_{\R^2} W (\BS W)
	\cdot \frac{\xi}{2 \pi \abs{\xi}^2 }
	\paren[\Big]{1 - 4 \pi G}
	\, d\xi.
\end{multline*}
The first term on the right integrates to $0$ (by equation~\eqref{eqnNonMagic}).
Further for any radial function (hence certainly for $W = W_s$) the second term vanishes.
Consequently,
%\begin{multline*}
%  - \int_{\R^2} W (\BS W) \cdot \frac{\nabla W_s}{W_s} \, d\xi
%  \\
%  = \beta \int_{\R^2} (W - W_s) \BS (W-W_s)
%	\cdot \frac{\xi}{2 \pi \abs{\xi}^2 }
%	\paren[\Big]{1 - 4 \pi G}
%	\, d\xi.
%\end{multline*}
\begin{multline*}
  \partial_\tau H =
    -\int_{\R^2} W
	\abs[\Big]{ \frac{ \nabla W}{W} - \frac{\nabla W_s}{W_s} }^2
       \, d\xi
  \\
  + \beta \int_{\R^2} (W - W_s) \BS (W-W_s)
	\cdot \frac{\xi}{2 \pi \abs{\xi}^2 }
	\paren[\Big]{1 - 4 \pi G}
	\, d\xi.
\end{multline*}
While the second term on the right should, in principle, be small (at least for small values of $\beta$ and when $W$ is close to $W_s$), we are (presently) unable to dominate this by the first term and show that $\partial_\tau H \leq 0$.
Thus we do not know whether the steady state $W_s$ is stable under large perturbations.

\subsection{Stability under small perturbations}\label{sxnRodrigues}%{{{1

We now turn to proving stability of $(\tilde \omega_\beta, \tilde d_\beta )$ as stated in theorem~\ref{thmBetaNonZero}.

\begin{proof}[Proof of theorem~\ref{thmBetaNonZero}]%{{{
Using the $\xi$-$\tau$ coordinates, let $(W,D)$ be solutions to the system~\sysWD with initial data $W_0, D_0 \in L^2_w$.
Define the perturbations from the steady state $D_p$, $U_p$ and $W_p$ by
\begin{equation}\label{eqnWpDpDef}
W_p \defeq W - \alpha W_s,
\quad
D_p \defeq D - \beta G,
\quad\text{and}\quad
U_p \defeq
\BS W_p + \gradinv W_p.
\end{equation}
In this setting, theorem~\ref{thmBetaNonZero} will follow if we establish
\begin{align}
  \label{eqnWpBound}
    \|W_p(\tau)\|_w &\leq C \paren[\big]{
	\|W_p(\tau_0)\|_w e^{-\gamma\tau} + \norm{D_p(\tau_0)}_{L^1(1)}e^{-\tau/2}
      }
  % \\
  % \label{eqnDpBound}
  % \llap{\text{and}\qquad}
  % \|D_p(\tau)\|_w &\leq C \|D_p(0)\|_w e^{-\gamma\tau},
\end{align}
for some constant $C$, where $\tau_0 = \log(t_0)$.
As before, the estimate for $D$ in theorem~\ref{thmBetaNonZero} is analogous to lemma~\ref{d_estimates}.
% Note, since $D_p(0) \in L^2_w$, we certainly have $D_p(0) \in L^1(1)$ and so the right hand side of~\eqref{eqnWpBound} is finite.
% The estimate for $D_p$ is much simpler than that for $W_p$ as the evolution for $D_p$ is simply the linear heat equation.
% It follows, for instance, by using exactly the same method we use to estimate $W_p$, and in the interest of brevity we only prove the bound for $W_p$ stated in~\eqref{eqnWpBound}.

To begin we state one basic result without proof.  First, a straightforward adaptation of the work in~\cite{bblRodrigues09}*{theorem~1} yields the following existence result.
\begin{lemma}\label{lmaExistence}
For $\epsilon_0>0$, there exists $\delta_0>0$, depending only on $\alpha$, such that if $W(0), D(0) \in L_w^2$ and
\[
	|\beta| + \|W_p(\tau_0)\|_w + \|D_p(\tau_0)\|_w
		\leq \delta_0,
\]
then there is a unique solution to the system~\sysWD such that, for all $\tau$, 
\begin{equation}\label{eqnWpDpSmall}
  \|D_p(\tau)\|_w + \|W_p(\tau)\|_w \leq \epsilon_0.
\end{equation}
\end{lemma}

In order to show convergence to the steady state, we work with the equation for the perturbation,
\begin{equation}\label{eqnPerturbative}
  \partial_\tau W_p + \dv \paren[\big]{ U W_p + \alpha \BS W_p W_s + \alpha \gradinv D_p W_s} = \mathcal L W_p.
\end{equation}
We multiply~\eqref{eqnPerturbative} by $G^{-1}W_p$ and integrate to obtain
\begin{multline}\label{eqnWpInt}
  \frac{1}{2} \partial_\tau \norm{W_p}_w^2
    + \int_{\R^2} G\inv W_p
	\dv \paren[\big]{ U W_p + \alpha \BS W_p W_s + \alpha \gradinv D_p W_s}
  \\
      = \int_{\R^2} G\inv W_p \cL W_p.
\end{multline}

We estimate each term individually.
First, for the right hand side, we use a coercivity estimate proven in~\cite{bblRodrigues09}.
Namely, since $\int W_p \, d\xi = 0$,
%Hence $\int G\inv W' \mathcal L W' \geq 1/2 \int G\inv (W')^2$ for any mean-zero function $W'$.
for any $\gamma \in (0, 1/2)$ and $\epsilon>0$ such that $\gamma + 1000\epsilon < 1/2$, we have
\begin{equation}\label{eqnLSpectral}
  -\int G^{-1} W_p \cL W_p
	  \geq (\gamma+\epsilon) \|W_p\|_w^2 + \frac{1-2(\gamma+\epsilon)}{2}\left[\frac{1}{3}\|\nabla W_p\|_w^2 + \frac{1}{32}\|\xi W_p\|_w^2\right].
\end{equation}
This is proved by first observing operator $L \defeq -G^{-1/2} \mathcal L G^{1/2}$ is a harmonic oscillator with spectrum $\set{ 0, 1/2, 1, 3/2, \dots }$ where $0$ is a simple eigenvalue.
Combining this with a standard energy estimate shows~\eqref{eqnLSpectral}, and we refer the reader to~\cite{bblGallayWayne02}*{Appendix A} or~\cite{bblRodrigues09}*{\S3.1} for the details.
We assume, without loss of generality, that $\gamma > 1/4$.

For the first term in the integral on the left of~\eqref{eqnWpInt}, observe
\begin{align}
  \MoveEqLeft
  \nonumber
  \int_{\R^2} G\inv W_p \dv (U W_p) \, d\xi
    = \int_{\R^2} \paren[\big]{ G\inv W_p^2 D + \frac{1}{2} G\inv U  \cdot \nabla \paren[\big]{W_p^2} } \, d\xi
  \\
  \nonumber
    &= \frac{1}{2} \int_{\R^2} G\inv W_p^2 \paren[\big]{D - \frac{1}{2} \xi \cdot U } \, d\xi 
  \\
  \label{eqnWpInt1}
    &=
      \frac{1}{2} \int_{\R^2} G\inv W_p^2 \paren[\big]{D - \frac{1}{2} \xi \cdot \gradinv D } \, d\xi 
      + \int_{\R^2} G\inv W_p (\BS W_p) \cdot \nabla W_p \, d\xi,
\end{align}
since $\BS W_s \cdot \xi = 0$.

To estimate this we claim
\begin{gather}
  \label{eqnDbound}
    \|D\|_w + \|D\|_{L^\infty}
      + \|\gradinv D\|_{L^\infty}
      \leq C \left[ \abs{\beta} + \|D_p(0)\|_w\right],
  \\
  \label{eqnBSbound}
  \llap{\text{and}\qquad}
  \norm{\BS W_p}_{L^\infty} \leq C \paren[\big]{ \norm{W_p}_w + \norm{\nabla W_p}_w },
\end{gather}
for some constant $C$ that is independent of $D_0, W_p$ and $\beta$.
To avoid breaking continuity we defer the proof of these estimates to section~\ref{sxnDivBounds} and continue with our proof of theorem~\ref{thmBetaNonZero} here.

Let $\epsilon_0$ to be a small constant to be determined later.
Using lemma~\ref{lmaExistence}, choose $\delta_0$ to guarantee~\eqref{eqnWpDpSmall} holds.
Then, returning to~\eqref{eqnWpInt1} we see
\begin{equation*}
  \abs[\Big]{\int_{\R^2} G\inv W_p \dv (U W_p) \, d\xi}
  \leq
    C \paren[\big]{ \abs{\beta} + \norm{D_p(\tau_0)}_w + \epsilon_0 }
      \paren[\big]{ \norm{W_p}_w^2 + \norm{\nabla W_p}_w^2 }.
\end{equation*}

For the second term in the integral on the left of~\eqref{eqnWpInt} we obtain smallness by using the fact that this term vanishes when $W_s = G$.
Indeed,
\begin{equation*}
  \alpha \int_{\R^2} G\inv W_p \, \BS W_p \cdot \nabla W_s \, d\xi 
    = -\alpha \int_{\R^2} G\inv W_s \BS W_p
	  \paren[\big]{ \nabla W_p + \frac{\xi}{2} W_p } \, d\xi,
\end{equation*}
which vanishes when $W_s = G$ due to the identity~\eqref{eqnNonMagic}.
Consequently,
\begin{multline}\label{eqnWpInt2}
  \alpha \int_{\R^2} G\inv W_p \, \BS W_p \cdot \nabla W_s \, d\xi 
    \\
    = -\alpha \int_{\R^2} G\inv (W_s - G) \BS W_p
	  \paren[\big]{ \nabla W_p + \frac{\xi}{2} W_p } \, d\xi.
\end{multline}
We claim that for all $\beta$ sufficiently small,
\begin{equation}\label{eqnWsMinusG}
  \norm{W_s - G}_w \leq C \abs{\beta},
\end{equation}
for some universal constant $C$.
Again, to avoid breaking continuity, we defer the proof of~\eqref{eqnWsMinusG} to section~\ref{sxnDivBounds}, and continue with the proof theorem~\ref{thmBetaNonZero}.

Equations~\eqref{eqnWpInt2} and~\eqref{eqnWsMinusG} immediately show
\begin{align}
  \MoveEqLeft\nonumber
  \abs[\Big]{ \alpha \int_{\R^2} G\inv W_p \, \BS W_p \cdot \nabla W_s \, d\xi }
  \\
  \nonumber
  &\leq
    C \abs{\alpha \beta}
	\norm{\BS W_p}_{L^\infty}
	\norm{W_p}_w \norm{\nabla W_s}_w
  \\
  %\nonumber
  %&\leq
  %  C \abs{\alpha \beta}
  %      \paren{\norm{W_p}_{w} + \norm{\nabla W_p}_w}
  %      \norm{W_p}_w
  %\\
  \label{eqnWpInt22}
  &\leq
    C \abs{\alpha \beta}
      \paren{
	\norm{W_p}_w^2 + \norm{\nabla W_p}_w^2
      }.
\end{align}
For the last inequality above we absorbed $\norm{\nabla W_s}_w$ into the constant $C$, and used~\eqref{eqnBSbound} and interpolation.

For the last term in the integral on the left of~\eqref{eqnWpInt} observe
\begin{align*}
  \abs[\Big]{\alpha \int_{\R^2}
    G\inv  W_p 
    \dv ( \gradinv D_p W_s ) \, d\xi}
  &= \abs[\Big]{\alpha \int_{\R^2}
      G\inv W_p (D_p W_s + \gradinv D_p \cdot \nabla W_s) \, d\xi}
  \\
  &\leq
    \abs{\alpha} \norm{W_p}_w \paren[\big]{\norm{D_p}_{L^\infty} + \norm{\gradinv D_p}_{L^\infty}} \norm{W_s}_w
  \\
  & \leq
    C \abs{\alpha} \norm{W_p}_w
      \paren[\big]{
	\norm{D_p}_{L^\infty}
	+ \norm{D_p}_{L^1}^{1/2} \norm{D_p}_{L^\infty}^{1/2}
      }.
\end{align*}
The last estimate followed from the interpolation inequality
\begin{equation}\label{eqnGradInvDBd}
  \norm{\gradinv D_p}_{L^\infty}
    \leq C \norm{D_p}_{L^1}^{1/2} \norm{D_p}_{L^\infty}^{1/2},
\end{equation}
the proof of which can be found in~\cite{bblRodrigues09} or~\cite{bblGallayWayne02} (see also proposition~\ref{kernel_bounds} in section~\ref{sxnCompactness}, below).

Since $D_p$ satisfies~\eqref{eqnD} with mean-zero initial data $D_p(\tau_0) \in L^1(1)$, it must satisfy the decay estimate~\eqref{eqnDfastDecay}.
Thus
\begin{align*}
  \abs[\Big]{\alpha \int_{\R^2}
    G\inv  W_p 
    \dv ( \gradinv D_p W_s ) \, d\xi}
  &\leq C \norm{D_p(\tau_0)}_{L^1(1)} e^{-\tau/2} \norm{W_p}_w
  \\
  &\leq \frac{\epsilon}{8} \norm{W_p}_w^2 + C \norm{D_p(\tau_0)}_{L^1(1)}^2 e^{-\tau}.
\end{align*}

Making $(1 + \abs{\alpha})\abs{\beta}$, $\delta_0$ and $\epsilon_0$ small enough, our estimates so far give
\begin{multline*}
	\frac{1}{2} \partial_\tau \|W_p\|_w^2 + (\gamma+\epsilon) \|W_p\|_w^2 + \frac{1 - 2(\gamma+\epsilon)}{2} \left[\frac{1}{3}\|\nabla W_p\|_w^2 + \frac{1}{32} \|\xi W_p \|_w^2\right]\\
	\leq
	  \epsilon\left[\|W_p\|_w^2 + \|\xi W_{p, k+1}\|_w^2 + \|\nabla W_p\|_w^2\right]
	  + C e^{-\tau} \norm{D_p(\tau_0)}_{L^1(1)}^2.
\end{multline*}
Because we chose $\epsilon$ small enough, the first three terms on the right can be absorbed in the left.
Consequently,
\begin{equation*}
  \partial_\tau \|W_p(\tau)\|_w^2 + 2\gamma \|W_p\|_w^2
    % + \frac{\epsilon}{4}\left[\frac{1}{3}\|\nabla W_p\|_w^2 + \frac{1}{32} \|\xi W_p \|_w^2\right]
  \leq C e^{-\tau} \norm{D_p(0)}_{L^1(1)}^2,
\end{equation*}
which immediately implies~\eqref{eqnWpBound}.
\end{proof}%}}}

\subsection{Proofs of estimates}\label{sxnDivBounds}
In this section, we prove the bounds~\eqref{eqnDbound}, \eqref{eqnBSbound} and \eqref{eqnWsMinusG}, which were used in the proof of theorem~\ref{thmBetaNonZero}.
We begin with the bounds on the divergence.
\begin{lemma}%{{{
Let $D$ satisfy~\eqref{eqnD} with initial data $D_0 \in L^2(w)$, and let $\beta = \int D_0 d\xi$.
Then if $D_p = D - \beta G$, there exists a uniform constant $C>0$ such that~\eqref{eqnDbound} holds.
\end{lemma}%}}}
\begin{proof}%{{{
Multiplying~\eqref{eqnD} by $G^{-1}D$, integrating and using the coercivity estimate~\eqref{eqnLSpectral} gives
\[
	\frac{1}{2}\partial_\tau \|D\|_w^2
		+ \frac{1}{4}\left[ \|D\|_w^2 + \frac{1}{3} \|\nabla D\|_w^2 + \frac{1}{32} \|\xi D\|_w^2 \right] \leq 0.
\]
Integrating this inequality in $\tau$ gives us the desired inequality for $\|D\|_w$.

Further, in the standard $x$-$t$ coordinates, $D$ solves the heat equation.
The classical estimates for solutions to the heat equation give us
\[
	\|D(\tau)\|_{L^\infty} + \|D(\tau)\|_{L^1} \leq C\|D(\tau_0)\|_{L^1} \leq C \|D(\tau_0)\|_w.
\]
Combined with the interpolation inequality~\eqref{eqnGradInvDBd} this yields the same bound for $\norm{\gradinv D}_{L^\infty}$, completing the proof.
\end{proof}%}}}

%\begin{remark}
%We could obtain more precise estimates by using that $D_p$, in the original variables, satisfies the heat equation and has mean zero.  However, these do not improve our results, so we omit them.
%\end{remark}

Now we turn to~\eqref{eqnBSbound}, which follows using the Sobolev embedding theorem and interpolation.
\begin{proof}[Proof of inequality~\eqref{eqnBSbound}]%{{{
%  \sidenote{Fri 11/14 GI: Perhaps we should just incorporate this after~\eqref{eqnBSbound} directly.}
  We know that the Biot-Savart operator satisfies the interpolation inequality
  \begin{equation*}
    \norm{\BS W_p}_{L^\infty} \leq C \norm{W_p}_{L^{4/3}}^{1/2} \norm{W_p}_{L^4}^{1/2}.
  \end{equation*}
  The proof is the same as that of~\eqref{eqnGradInvDBd}, and can be found in~\cites{bblRodrigues09, bblGallayWayne02} (see also proposition~\ref{kernel_bounds} in section~\ref{sxnCompactness}, below).
  Combining this with Sobolev inequality we obtain
  \begin{multline*}
    \norm{\BS W_p}_{L^\infty}
      \leq C \norm{W_p}_{L^{4/3}}^{1/2} \norm{W_p}_{L^4}^{1/2}
      \leq C \norm{W_p}_{L^{4/3}}^{1/2} \norm{\nabla W_p}_{L^{4/3}}^{1/2}
    \\
      \leq C \norm{W_p}_{L^2(w)}^{1/2} \norm{\nabla W_p}_{L^2(w)}^{1/2}
      \leq C \paren[\big]{\norm{W_p}_{L^2(w)} + \norm{\nabla W_p}_{L^2(w)} },
  \end{multline*}
  as desired.
\end{proof}%}}}

Finally, we prove~\eqref{eqnWsMinusG} showing $W_s$ is close to $G$ when $\beta$ is small.
\begin{lemma}%{{{
Let $W_s\in L_w^2$ be a solution to equation~\eqref{eqnWs}.
%\[\cL W_s = \beta G + \beta\gradinv G \cdot \nabla W_s,\]
Then there is a universal constant $C > 0$ such that  such that the inequality~\eqref{eqnWsMinusG} holds for all $\beta$ sufficiently small.
\end{lemma}%}}}
\begin{proof}
%\sidenote{Fri 11/14 GI: WOAH! This is overkill. We should have a trivial ODE argument based on~\eqref{eqnWs}. But in the interest of time, I'm willing to let it slide.}
Define $P_s = W_s - G$.  Notice that this solves
\[ \cL P_s = \beta G P_s + \beta\gradinv G \cdot \nabla P_s + \beta G^2 + \beta \gradinv G \cdot \nabla G.\]
Multiply this equation by $G^{-1} P_s$ and using~\eqref{eqnLSpectral}, with $\gamma = 1/4$, to obtain that
\[\begin{split}
\frac{1}{4} \|P_s\|_w^2 + \frac{1}{4}\left[\frac{1}{3}\|\nabla P_s\|_w^2 + \frac{1}{32}\|\xi P_s\|_w^2\right]
	&\leq - \int G^{-1} P_s \cL P_s\\
	&= - \beta \int P_s^2 - \beta \int G^{-1} P_s \gradinv G\cdot \nabla P_s  \\
	& -\beta \int GP_s - \beta\int G^{-1} P_s \gradinv G \cdot \nabla G\\
	&\leq (2|\beta|+\epsilon)\left[ \|P_s\|_w^2 + \|\nabla P_s\|_w^2\right] + |\beta|^2C_\epsilon
.\end{split}\]
%Here we used Young's inequality.
Here $\epsilon < 1/20$ is a positive constant.
Then when $\beta$ is sufficiently small, we may absorb the terms on the last line into the left hand side, giving~\eqref{eqnWsMinusG} as desired.
\end{proof}

\section{Bounds for the vorticity}\label{sxnInequalities}

Bounds on the vorticity to the standard 2D incompressible Navier-Stokes equations are well known.
In this section we prove the analogues of these bounds for the extended Navier-Stokes equations~\eqref{eqnExtendedNS1}.

We begin with the vorticity decay in $L^p$.
The strategy for this proof is not entirely different from the classical case, however, the appearance of a divergence term complicates matters and yields a slightly different final estimate.
We will use this estimate in the proof of~\eqref{eqnWbound} and in our discussion of well-posedness in section~\ref{sxnWellPosed}.

\begin{lemma}\label{e_vorticity_decay}
Let $p$ be an element of $[1,\infty]$, and suppose that $(\omega, d)$ solve the system~\sysOmegaD with $\omega_0, d_0 \in L^1$.  Then there exists $C>0$, depending only on $p$, $\|\omega_0\|_{L^1}$, and $\|d_0\|_{L^1}$ such that
\begin{equation}\label{vort_decay}
	\|\omega\|_{L^p} + t^{1/2}\|\nabla \omega\|_{L^p} \leq \frac{C}{t^{1-1/p}}
\end{equation}
and
\begin{equation}\label{vort_grad_decay}
	\|\nabla\omega\|_{L^p} \leq \frac{C}{t^{3/2 - 1/p}}.
\end{equation}
\end{lemma}
\begin{proof}
We omit the proof of the bound on the gradient.  Indeed, by following the work in~\cite{bblKato94}*{proposition~4.1}, we note that the estimate relies only on~\eqref{vort_decay} and Duhamel's principle.  In view of this, obtaining this result is a straightforward adaptation.

Now, we obtain the $L^p$ bound by obtaining a bound in $L^1$ and $L^\infty$ and interpolating.  The $L^1$ bound follows by splitting $\omega_0$ into its positive and negative parts, using the maximum principle, and using that the mass is preserved.

The classical technique for obtaining the $L^\infty$ bound has three steps: (i) get a bound on the $L^2$ norm in terms of the $L^1$ norm divided by $t^{1/2}$, (ii) show that this gives a bound on the $L^\infty$ norm in terms of the $L^2$ norm divided by $t^{1/2}$ for the adjoint problem, and (iii) apply these inequalities over $[0,t/2]$ and $[t/2,t]$ to finish.  Since the work in (ii) is the same as the work in (i) and since (iii) is unchanged from the classical setting, we simply show the first step (i).  To this end, multiplying our equation by $\omega$ and integrating by parts gives us
\[
	\frac{d}{dt} \|\omega \|_{L^2}^2 \leq  -2\| \nabla \omega\|_{L^2}^2 + 2\|d\|_{L^\infty} \|\omega\|_{L^2}^2
.\]
Using the Fourier transform, we see that there is a constant $C>0$ such that for any $R$,
\[\begin{split}
	\|\hat\omega\|_{L^2}^2 &\leq \int_{B_R^c} \frac{|\xi|^2}{R^2}|\hat\omega|^2d\xi + \int_{B_R} |\hat\omega|^2 d\xi\\
		&\leq \frac{1}{R^2} \int |\xi|^2 |\omega|^2 d\xi + \int_{B_R} \|\hat\omega\|_{L^\infty}^2 d\xi\\
		&\leq \frac{1}{R^2} \|\nabla\omega\|_{L^2}^2 + C R^2 \|\omega\|_{L^1}^2.
\end{split}\]
Using $R = \|\omega\|_{L^2}^2/(2C\|\omega_0\|_{L^1}^2)$ along with these inequalities  yields
\begin{equation}\label{e_diff_inequality}
	\frac{d}{dt} \|\omega\|_{L^2}^2
		\leq \left[ \frac{C \|d_0\|_{L^1}}{t} - \frac{\|\omega\|_{L^2}^2}{2C \|\omega_0\|_{L^1}^2} \right] \|\omega\|_{L^2}^2\\
.\end{equation}
Here we used the standard estimates for the heat equation, and then we used Young's inequality.  Define $\phi(t) = t \|\omega\|_{L^2}^2$ to obtain
\[
	\phi'(t) \leq \frac{\phi}{t}\left[\|d_0\|_{L^1} - \frac{\phi}{2C \|\omega_0\|_{L^1}^2} + 1 \right].
\]
This implies that $\phi \leq 2C \|\omega_0\|_{L^1}^2[ \|d_0\|_{L^1}^2 + 1]$, which proves our claim.
\end{proof}

Now, we prove the pointwise, heat kernel type bound on the vorticity when $\beta = 0$ stated in lemma~\ref{lmaCompactness}.  We use the increased decay of the heat equation when the initial data is mean-zero here.  The key point here is that the $L^\infty$ norm of the divergence is integrable in time, so we may reproduce the classical arguments in this case.  We follow the work of Carlen and Loss in~\cite{bblCarlenLoss95} in order to do this.

%\begin{lemma}\label{heat_kernel}
%Suppose that $\omega$ and $d$ satisfy~\sysOmegaD with $\omega_0\in L^1$ and $|x|d_0,$ $d_0 \in L^1(1)$.  Then there exists $C$ depending on $\|\omega_0\|_{L^1}$ and $\|(1+|\cdot|)d_0\|_{L^1}$ such that
%\[
%	|\omega(t,x)| \leq \frac{C}{t} \int_{\R^2} \exp\left(-\frac{|x-y|^2}{Ct}\right) |\omega_0(y)|dy.
%\]
%\end{lemma}
\begin{proof}[Proof of~\eqref{eqnWbound}]
Our first step is to obtain bounds for the equation
\begin{equation}\label{kernel_linear_problem}
	\phi_t = \Delta \phi + \nabla\cdot(b \phi) + c \phi.
\end{equation}
which depend only on certain norms of $b$ and $c$.  To this end, fix $T>0$ and we let $r(t)$ be a monotone increasing, smooth function defined on $[0,T]$ to be determined later.  In addition, we may assume without loss of generality that $\phi$ is non-negative.  Then we calculate
\[\begin{split}
r(t)^2 \|\phi\|_{L^r}^{r-1} \frac{d}{dt}\|\phi\|_{L^r}
	&= \dot{r} \int \phi^r \log\left( \frac{\phi^r(x)}{\|\phi\|_{L^r}}\right) dx + r(t)^2 \int \phi^{r-1}\phi_t dx\\
	&= \dot{r} \int \phi^r \log\left( \frac{\phi^r(x)}{\|\phi\|_{L^r}}\right) dx\\
	  &\qquad + r(t)^2 \int \phi^{r-1}\left(\Delta \phi + \nabla\cdot(b\phi) + c\phi\right)dx\\
	&=  \dot{r} \int \phi^r \log\left( \frac{\phi^r(x)}{\|\phi\|_{L^r}}\right) dx + 4(r-1) \int \left| \nabla\left(\phi^{r/2}\right)\right|^2 dx\\
	&\ \ \ \  + \int r(r-1) \phi^r \left(\nabla \cdot b\right) + r^2 \int c\phi^r dx.
\end{split}\]
The log-Sobolev inequality~\cite{bblCarlenLoss95}*{Equation (1.17)}, which the authors derive from the work in~\cite{bblGross75}, is
\begin{equation}\label{lsi}
\int |f|^2 \log\left(\frac{f^2}{\|f\|_{L^2}^2}\right) dx + (2 + \log(a))\int |f|^2 dx
	\leq \frac{a}{\pi}\int |\nabla f|^2 dx,
\end{equation}
for any $f \in H^1$ and $a\in (0,\infty)$.  Applying this with $a = 4\pi (r-1)/\dot{r}$, gives us
\[\begin{split}
r(t)^2 \|\phi\|_{L^r}^{r-1} \frac{d}{dt}\|\phi\|_{L^r} \leq& - \dot{r} \left(2 + \log\left( \frac{4\pi(r-1)}{\dot{r}}\right)\right) \|\phi\|_{L^r}^r\\
	& + \left(r(r-1) B(t) + r^2 C(t)\right)\|\phi\|_{L^r}^r
,\end{split}\]
where $B(t) = \|\nabla\cdot b(t,\cdot)\|_{L^\infty}$ and $C(t) = \|c(t,\cdot)\|_{L^\infty}$.  Now we set $G(t) = \log \|\phi\|_{L^r}$ and $s = 1/r$ to obtain
\[
	\frac{dG}{dt} \leq \dot{s}\left(2 + \log(4\pi s(1-s))\right) - \dot{s}\log\left(-\dot{s}\right) + (1-s)B(t) + C(t)
.\]
Letting $s(t)$ be a linear interpolation of $1$ and $0$ over $[0,T]$, we see that $\dot{s} = -T^{-1}$.  Then we may integrate this to obtain
\[
	G(T) - G(0)
		\leq 4 - \log(4\pi) - \log(T) + \int_0^T \left[B(t) + C(t)\right] dt
.\]
Exponentiating gives us
\begin{equation}\label{preliminary_heat_bound}
	\|\phi(T)\|_{L^\infty}
		\leq \frac{K}{T} \exp\left( \int_0^T \left[B(t) + C(t)\right] dt\right)
.\end{equation}

In order to get pointwise decay from \eqref{preliminary_heat_bound}, we look at the operator
	\[P^{(\alpha)}(T,x,y) := e^{-\alpha \cdot x} P(T,x,y) e^{\alpha\cdot y},\]
where $P$ is the solution kernel for our linear problem \eqref{kernel_linear_problem} with $c\equiv 0$ and $\alpha(x,y)$ is a function to be identified later.  We assume that $b$ can be written as $b = b_1+b_2$ where
$\nabla\cdot b_1 = 0$
\begin{equation}\label{eqnB}
  \|b_1(t)\|_{L^\infty} \leq \frac{K_1}{\sqrt{t+1}},
  \quad
  \|\nabla\cdot b_2\|_{L^\infty} \leq \frac{K_2}{(t+1)^{3/2}},
  \quad\text{and}\quad
  \|b_2\|_{L^\infty} \leq \frac{K_2}{(t+1)}.
\end{equation}
In the application we have in mind, $b_1$ comes from the Biot-Savart kernel of the vorticity, while $b_2$ comes from $\gradinv$ of the divergence.

We wish to obtain bounds for $P$ through our integral bounds on $P^{(\alpha)}$.  To this end, we notice that $P^{(\alpha)}$ is the solution kernel for the problem
\[
	\phi_t = \Delta \phi + \nabla\cdot(( b + 2\alpha) \phi ) + (\alpha\cdot b + |\alpha|^2) \phi.
\]
Applying \eqref{preliminary_heat_bound}, and noticing that $\nabla\cdot(b + 2\alpha) = \nabla\cdot b$, we obtain, for any $\alpha$,
\[\begin{split}
	P^{(\alpha)}&(T,x,y) \leq\\
		 &\frac{K}{T} \exp\left( 2 \int_0^T \left[ K_2(t+1)^{-3/2} + K_2^2(t+1)^{-2} + |\alpha| K_1 (t+1)^{-1/2} + |\alpha|^2\right] dt\right).
\end{split}\]
Choosing
\[
	\alpha = - \frac{1}{4T}\frac{(x-y)}{|x-y|} \left[ |x-y| - 2 K_1\sqrt{T+1}\right]_+,
\]
using the definition of $P^{(\alpha)}$, and integrating in time, we obtain
\[
	P(T,x,y)
		\leq \frac{K}{T} \exp\left( (4K_2 + 2K_2^2)  - \frac{1}{8T} \left[|x-y| - 2K_1\sqrt{T+1} \right]_+^2 \right).
\]
By possibly changing the constants, we may obtain
\[
	P(T,x,y)
		\leq \frac{C}{T} \exp\left( - \frac{|x-y|^2}{CT}\right).
\]

To conclude, we apply the above to equation~\eqref{eqnENSOmega}, by choosing $b_1 = \BS \omega$ and $b_2 = \gradinv d$.
Lemmas~\ref{d_estimates} and~\ref{e_vorticity_decay} and interpolation inequalities of the form~\eqref{eqnGradInvDBd} show that~\eqref{eqnB} is satisfied, concluding the proof.
\end{proof}

  % !TEX root = oseen.tex

\section{Relative compactness of complete trajectories}\label{sxnCompactness}

In this section we prove lemmas~\ref{lmaCompactness} and~\ref{lmaWCompact}, showing that complete trajectories in $L^1$ are relatively compact.
The development is similar to \cite{bblGallayWayne05}, and the main difference here is the additional divergence term which requires us to alter many of the proofs.
%\highlight{for the estimates in Section~\ref{sxnInequalities}.}\sidenote{Forward reference.}
%To prove lemma~\ref{lmaCompactness} we first show that $\{W(\tau)\}_{\tau\in \R}$ is relatively compact in $L^2(m)$, then establish that $\{W(\tau)\}_{\tau\geq 0}$ is relatively compact in $L^1$.
We first work up towards proving lemma~\ref{lmaWCompact}, and then use this to prove lemma~\ref{lmaCompactness}.

\subsection{The semi-group of \texorpdfstring{$\mathcal L$}{L} and apriori bounds.}

In order to obtain the desired compactness results, we will need estimates on various quantities. We will state these estimates here, but we will omit the proofs and provide references.

Let $S(\tau) \defeq \exp\paren{\tau \cL}$ be the semigroup generated by the operator $\cL$.  First we recall some estimates on the operator $S(\tau)$.  In order to state these, we define the function
\[a(\tau) \defeq 1 - e^{-\tau}.\]
This function appears naturally with the change of variables.  We recall a lemma on the operator $S$ from \cite{bblGallayWayne02}.
\begin{lemma}\cite{bblGallayWayne02}*{Appendix~A}\label{SemiGroupBds}
\begin{enumerate}
\item For $m>1$, $S(\tau)$ is a bounded operator on $L^2(m)$.  In addition, $\nabla S(\tau)$ is bounded away from $\tau = 0$.  More precisely, there is a universal constant $C$ such that
\[ \|S(\tau)\|_{L^2(m)\to L^2(m)} \leq C, \ \ \|\nabla S(\tau)\|_{L^2(m)\to L^2(m)} \leq \frac{C}{\sqrt{a(\tau)}}.\]
\item Let $L_0^2(m)$ be the space of $L^2(m)$ functions with integral zero.  For $\mu \in (0,1/2]$ and $m > 1 + 2\mu$ and $\tau > 0$, there is a universal constant $C$ such that
%\sidenote{Thu 10/30 GI: What is $L^2_0(m)$?}
\[\|S(\tau)\|_{L^2_0(m)\to L^2_0(m)} \leq C e^{-\mu \tau}, \ \ \|\nabla S(\tau)\|_{L^2_0(m)\to L^2_0(m)} \leq  C\frac{e^{-\mu \tau}}{\sqrt{a(\tau)}}.\]
\item For $1 \leq q \leq p \leq \infty$, $T>0$, $m \in [0,\infty)$ and $\alpha \in \N^2$, there is a constant $C_T$, depending on $T$, such that
\[\|\partial^\alpha S(\tau) f\|_{L^p(m)}
	\leq \frac{C_T}{a(\tau)^{(q^{-1} - p^{-1})+ |\alpha|/2}} \|f\|_{L^q(m)},\]
for any $f \in L^q(m)$ and any $0 < \tau \leq T$.
\end{enumerate}
\end{lemma}
\noindent We note that the commutator of $\nabla$ and $S(\tau)$ is computed as
\[\partial_i S(\tau) = e^{\tau/2} S(\tau) \partial_i.\]

In addition, we need the well-known bounds on Biot-Savart kernel and $\nabla^{-1}$.  The proof of this proposition may be found in~\cite{bblRodrigues09}*{proposition 1} and~\cite{bblGallayWayne02}*{Appendix B}.
\begin{proposition}\label{kernel_bounds}
Denote by $K$ either the operator $\BS$ or the operator $\gradinv$.  Then the following inequalities hold for any $f$ such that the right hand side of each inequality is finite.
\begin{enumerate}
\item If $1 < p < 2 < q$ and $1 + q^{-1} - p^{-1} = 1/2$ then there is a constant $C$ such that
\[ \|Kf\|_{L^q} \leq C \|f\|_{L^p}.\]
\item If $1 \leq p < 2 < q \leq \infty$ and $0 < \theta < 1$ satisfy
\[\frac{\theta}{p} + \frac{1-\theta}{q} = \frac{1}{2},\]
then there is a constant $C$ such that
\[\|Kf\|_{L^\infty} \leq C \|f\|_{L^p}^\theta \|f\|_{L^q}^{1-\theta}.\]
\item There exists a constant $C_p>0$ depending only on $p$ such that if $p>1$ then
\[\|\nabla Kf\|_{L^p} \leq C_p \|f\|_{L^p}.\]
\item If $0 < m < 1$ and $q > 2$ then there is a constant $C_q$, depending only on $q$, such that
\[ \|Kf\|_{L^q(m-2/q)}  \leq C_q \|f\|_{L^2(m)}.\]
\end{enumerate}
\end{proposition}

Finally, we state an \textit{a priori} bound on solutions to the system~\sysWD.  The proof of this lemma is a straightforward adaptation of~\cite{bblGallayWayne05}*{lemma~2.1}.
\begin{lemma}\label{WeightedBound}
Suppose that $(W,D)$ solves the system~\sysWD in the space
\begin{equation*}
  C^0([0,T],L^2(m))\cap C^0((0,T],H^1(m))
\end{equation*}
with $W_0\in L^2(m)$ and $D_0\in L^2(m)$ as the initial conditions for $W$ and $D$ respectively.  Then there is a constant $C$ such that
\[
  \norm{W(\tau)}_{L^2(m)} + a(\tau)^{1/2}\|\nabla W(\tau)\|_{L^2(m)} \leq C.\]
\end{lemma}
%}}}

\subsection{Compactness in \texorpdfstring{$L^2(m)$}{L2(m)}.}

First we
show relative compactness of complete trajectories on $\R_+$ in $L^2(m)$.  This is accomplished by decomposing the remainder term into convenient functions, two of which decay to zero and one whose trajectory is relatively compact.

\begin{lemma}\label{WCompactForward}
Assume that $m>3$ and that $(W,D)\in C^0([0,\infty), L^2(m)^2)$ is a solution to the system~\sysWD, and is bounded in $L^2(m)$.  The trajectory $\{(W,D)\}_{\tau\in \R_{\geq 0}}$ is relatively compact in $L^2(m)$.
\end{lemma}
\begin{proof}
We work here with $W$ only, but the proof for $D$ is similar and simpler.  We define the remainder, $R$, to be such that $W = \alpha G + R$.  One can check that
\begin{equation*}
  \partial_\tau R
    =
      \cL R - \alpha \Lambda R - N(R) - \nabla\cdot(W \gradinv D).
\end{equation*}
where
\begin{equation*}
  \alpha \Lambda R \defeq \left(\alpha \BS G \cdot \nabla R + \alpha \BS R \cdot \nabla G\right)
  \quad\text{and}\quad
  N(R) \defeq \BS R\cdot \nabla R.
\end{equation*}
Hence we may write
%\[\begin{split}
%	R(\tau, \xi) &= S(\tau)R_0 - \underbrace{\int_0^\tau S(\tau-s) (\alpha \Lambda R(s) + N(R)(s) )ds}_{\defeq R_1} \\
%	&~~~+ \underbrace{\int_0^\tau S(\tau-s) \nabla\cdot(W(s) \gradinv D(s)) ds}_{\defeq R_2}
%.\end{split}\]
\begin{equation}\label{e_duhamel}
  R(\tau, \xi) = S(\tau)R_0 - R_1 - R_2
\end{equation}
where
\begin{align*}
  R_1 &\defeq \int_0^\tau S(\tau-s) (\alpha \Lambda R(s) + N(R)(s) ) \, ds\\
  \text{and}\quad
  R_2 &\defeq \int_0^\tau S(\tau-s) \nabla\cdot(W(s) \gradinv D(s)) \, ds.
\end{align*}
The first term tends to zero by part two of lemma \ref{SemiGroupBds} and the fact that $\int R_0 d\xi = 0$. %\sidenote{\textbf{[Landon: How does $\int R_0 = 0$ come into play?] [ Chris: added ``part two'']}}
It follows from the work in lemma~2.2 in \cite{bblGallayWayne05} that $R_1$ is bounded in $L^2(m+1)$, and, hence, is a relatively compact trajectory.  Thus, we need only show that $R_2$ tends to zero.

To this end, we use the first inequality in lemma~\ref{SemiGroupBds} to obtain
\[\begin{split}
\|R_2\|_{L^2(m)}
	&\leq \int_0^\tau e^{-\frac{1}{2}(\tau-s)} \|\nabla S(\tau-s) (W\gradinv D)(s)\|_{L^2(m)} ds\\
	&\leq C\int_0^\tau \frac{e^{-\frac{1}{2}(\tau-s)}}{\sqrt{a(\tau-s)}} \|(W\gradinv D)(s)\|_{L^2(m)} ds\\
	&\leq C\int_0^\tau \frac{e^{-\frac{1}{2}(\tau-s)}}{\sqrt{a(\tau-s)}} \|\gradinv D(s)\|_{L^\infty} \|W(s)\|_{L^2(m)} ds
.\end{split}\]
The results of lemma~\ref{d_estimates} and proposition~\ref{kernel_bounds} imply that $\|\gradinv D(s)\|_{L^\infty}$ tends to zero as $s$ tends to infinity.  Hence, we see that $\|R_2\|_{L^2(m)}$ tends to zero as $\tau$ tends to infinity.
\end{proof}

Now we will show relative compactness of complete trajectories in $L^2(m)$, i.e. we will prove lemma~\ref{lmaWCompact}.  Our method of proof will be similar to above.
%
%[snip] Statement of lemma~\ref{lmaWCompact} was here. [/snip]

\begin{proof}[Proof of lemma~\ref{lmaWCompact}]%{{{
Again we will look at $R$ as above and only work with $W$.  This time we will decompose $R$ as
\[\begin{split}
	R(\tau) &= S(\tau - \tau_0) R(\tau_0) - \int_{\tau_0}^\tau S(\tau - s)\left( \alpha \Lambda R(s) + N(R)(s)\right) ds\\
	&~~~ - \int_{\tau_0}^\tau S(\tau - s) \nabla\cdot(W(s) \gradinv D(s)) ds
,\end{split}\]
where $\tau_0<\tau$.  Since $R \in L^2_0(m)$, by construction, it follows from lemma~\ref{SemiGroupBds} that $S(\tau - \tau_0)R(\tau_0)$ tends to zero as $\tau_0$ tends to negative infinity. %\sidenote{\textbf{[Landon: How does this work?][Chris: Added ``from lemma 3.5'']}}
Hence we may write
\begin{equation*}
  R(\tau) = -R_1 - R_2,
\end{equation*}
where
\begin{align*}
  R_1 \defeq \int_{-\infty}^\tau S(\tau - s)\left( \alpha \Lambda R(s) + N(R)(s)\right) \, ds \\
  \text{and}\quad
  R_2 \defeq \int_{-\infty}^\tau S(\tau - s) \nabla\cdot(W(s) \gradinv D(s)) \, ds.
\end{align*}
As before, showing that $R_1$ is relatively compact is exactly as in \cite{bblGallayWayne05}.  Thus, we need only investigate $R_2$, which we handle similarly to the previous lemma.

We will show that $R_2$ is bounded in $L^2(m+r)$ for some $r>0$.  For any $q \in (1,2)$, lemma \ref{SemiGroupBds} gives us
\[
	\|R_2\|_{L^2(m+r)} \leq C \int_{-\infty}^\tau \frac{e^{-\frac{1}{2}(\tau-s)}}{a(\tau-s)^{1/q}} \|W \gradinv D\|_{L^q(m+r)} ds.
\]
H\"older's inequality implies that 
\[
	\|W \gradinv D\|_{L^q(m)} \leq \|W\|_{L^2(m)} \|\gradinv D\|_{L^{2q/(2-q)}(r)}.
\]
The first term is bounded due to the assumptions in the statement of the current lemma.  For the remaining term concerning the divergence $D$, we apply proposition~\ref{kernel_bounds} to see that, letting $\tilde m = r + (2-q)/q$, and choosing $r$ and $q$ such that $\tilde m \leq m$, 
%, letting $q = 3/2$,
\[
	\|\gradinv D\|_{L^{2q/(2-q)}(r)} \leq C \|D\|_{L^2(\tilde m)} \leq C \|D\|_{L^2(m)}.
\]

Hence $R_2$ is bounded in $L^2(m+r)$.  Lemma~\ref{SemiGroupBds} and lemma~\ref{WeightedBound} imply that $R_2$ is also bounded in $H^1(m)$, %\sidenote{\textbf{[Landon: This is not obvious.][Chris: added ``Lemma 3.5 and'']}}
so that Rellich's theorem, see e.g.~\cite{bblReedSimon78}*{theorem~XIII.65} %\sidenote{\textbf{[Landon: Same comment as for LaSalle's invariance. Should we cite this?] [Chris: added citation]}}
 implies that $R_2$ is relatively compact in $L^2(m)$, finishing the proof.
\end{proof}%}}}

In order to conclude, we need that bounded trajectories in $L^1$ are relatively compact.  In order to show this, one may reproduce the proof of~\cite{bblGallayWayne05}*{lemma~2.5} as it relies only on a pointwise estimate on $W$, which we recreate in~\eqref{eqnWbound}.  This yields the final lemma we need to prove the necessary compactness.

\subsection{Convergence in \texorpdfstring{$L^\infty$}{L-infty}}

In this section we prove convergence of $W$ to $\alpha G$ in $L^\infty$, as stated in theorem~\ref{thmBetaZero}.

\begin{lemma}\label{lmaLinf}%{{{
%Let $(\omega, d)$ solve the system~\sysOmegaD with $\omega_0$, $(1 + \abs{x})d_0 \in L^1(\mathbb{R}^2)$.
Let $(W, D)$ solve the system~\sysWD with initial data $(W_0, D_0)$ such that $W_0, D_0 \in L^1(1)$.
If $\alpha = \int W_0 \, d\xi$ and $\beta = \int D_0 \, d\xi = 0$, then
\begin{equation*}
	\lim_{\tau\to\infty} \left\| W - \alpha G\right\|_{L^\infty} = 0.
\end{equation*}
\end{lemma}%}}}

%Above, we have developed the semigroup theory that we need to conclude convergence of $W$ to $\alpha G$ in $L^\infty$ in the mean zero case.  We will do this here; thus, concluding proposition~\ref{ppnMeanZero}.  See Remark~\ref{r_LInfinityConv}.

\begin{proof}%{{{
Recall that we have shown that $W$ converges to $\alpha G$ in $L^p$ for all $p\in[1,\infty)$.  As in~\eqref{e_duhamel}, letting $R = W-\alpha G$, we may write an integral equation for $R$ using the semigroup $S$.  We will use this to show that $\|R\|_{L^\infty}$ tends to zero.  As above, $R$ satisfies
\[\begin{split}
	R(\tau)
		&= S(1)R(\tau-1) - \int_{\tau-1}^\tau S(\tau-s)\left(\alpha \Lambda R(s) + N(R)(s)\right)ds\\
		& - \int_{\tau-1}^\tau S(\tau-s) \nabla\cdot( W(s)\nabla^{-1}D(s))ds.
\end{split}\]
First, we use the third conclusion of lemma~\ref{SemiGroupBds} with $p = \infty$, $q =1$, $\alpha = 0$ and $m = 0$ on the first term.  Hence, we have that
\[
	\|S(1)R(\tau-1)\|_{L^\infty}
		\leq C \|R(\tau-1)\|_{L^1}.
\]
Since $\|R(\tau-1)\|_{L^1}$ tends to zero, then $\|S(1)R(\tau-1)\|_{L^\infty}$ tends to zero.  We may use this same strategy to deal with the rest of the terms.

First we look at
\[
	\Lambda R = (\BS G) \cdot R + (\BS R)\cdot \nabla G
		=\nabla \cdot( (\BS G)  R + (\BS R) G).
\]
Then lemma~\ref{SemiGroupBds}, implies that,
\[\begin{split}
	\norm[\Big]{\int_{\tau-1}^\tau S(\tau-s) \Lambda R(s) ds}_{L^\infty}
		&\leq \int_{\tau-1}^\tau (\|(\BS G) R\|_{L^1} + \|(\BS R) G\|_{L^1}) ds\\
		&\leq C \int_{\tau-1}^\tau ( \|R\|_{L^1} + \|\BS R\|_{L^\infty})ds.
\end{split}\]
Since $R$ tends to zero in $L^p$ for all $p$, then lemma~\ref{kernel_bounds} implies that $\BS R$ tends to zero in $L^\infty$.

Next, we deal with the term involving $N(R)$.  Notice that $N(R) = \nabla\cdot((\BS R) R)$.  Hence, as above, we obtain
\[\begin{split}
	\norm[\Big]{\int_{\tau-1}^\tau S(\tau-s) N(R)(s) ds}_{L^\infty}
		&\leq \int_{\tau-1}^\tau \|(\BS R) R\|_{L^1} ds\\
		&\leq C \int_{\tau-1}^\tau \|\BS R\|_{L^\infty} \|R\|_{L^1} ds.
\end{split}\]
Hence, this term tends to zero as well.

Finally, for the last term, we obtain
\[\begin{split}
	\norm[\Big]{\int_{\tau-1}^\tau S(\tau-s) \nabla\cdot(W\nabla^{-1}D)(s) ds}_{L^\infty}
		&\leq \int_{\tau-1}^\tau \|W \nabla^{-1}D\|_{L^1} ds\\
		&\leq C \int_{\tau-1}^\tau \|\nabla^{-1}D\|_{L^\infty} \|W\|_{L^1} ds.
\end{split}\]
Using lemma~\ref{d_estimates} and lemma~\ref{kernel_bounds}, we see that $\|\nabla^{-1}D\|_{L^\infty}$ tends to zero.  This finishes the proof that $\|R\|_{L^\infty}$ tends to zero.
\end{proof}%}}}

\section{Brief Remarks on Well-posedness}\label{sxnWellPosed}

The well-posedness of the system~\sysOmegaD in classical or Lebesgue spaces is very similar to the development in~\cites{bblBen-Artzi94, bblBrezis94, bblKato94}.  For the weighted spaces, one may look to the strategies of~\cites{bblGallayWayne02, bblRodrigues09}.  Since the adaptations required in our setting are minimal, we only briefly comment on the manner of proof.  First, we discuss the primary a priori estimates in each of these spaces.  Then, we discuss the iterative scheme used to prove local existence.

\subsection*{A Priori Estimates}
The main a priori estimates in $L^p$ and in $L^2_w$ follow as in the work of lemma~\ref{e_vorticity_decay} and section~\ref{sxnNonMeanZero}, respectively.  The a priori estimate in $L^2(m)^2$  is a slight modification of the argument of \cites{bblGallayWayne02}.  To this end, multiply~\eqref{eqnW} by $|\xi|^{2m}W$ to obtain
\[\begin{split}
	\frac{1}{2}\frac{d}{d\tau} &\int |\xi|^{2m}W^2 d\xi + \int |\xi|^{2m} W \nabla\cdot(UW)d\xi\\
		& = \int |\xi|^{2m}\left\{W\Delta W + \frac{W}{2}(\xi\cdot\nabla)W + W^2\right\}d\xi.
\end{split}\]
Integrating by parts, we see that these terms can be rewritten as
\[\begin{split}
&\int |\xi|^{2m} W(\Delta W)d\xi = -\int |\xi|^{2m}|\nabla W|^2 d\xi + 2m(m-1)\int |\xi|^{2m-2}W^2 d\xi ,\\
&\int |\xi|^{2m}\frac{W}{2}(\xi\cdot\nabla )W d\xi = -\frac{m+1}{2}\int |\xi|^{2m} W^2d\xi,\\
&\int |\xi|^{2m} W\nabla \cdot(UW)d\xi = \frac{1}{2}\int |\xi|^{2m}DW^2 d\xi + \frac{1}{2}\int |\xi|^{2m}\nabla\cdot (UW^2)d\xi\\
&\quad\quad= \frac{1}{2}\int |\xi|^{2m}DW^2d\xi - m\int |\xi|^{2m-2}(\xi\cdot U)W^2d\xi.
\end{split}\]
By noting that for any $\epsilon>0$ there is a $C_\epsilon>0$ so that $|\xi|^{2m-2}\leq \epsilon|\xi|^{2m} + C_\epsilon$, we see that
\[\begin{split}
\frac{1}{2}\frac{d}{d\tau}\int |\xi|^{2m} W^2d\xi
	&+ \int |\xi|^{2m} |\nabla W|^2 d\xi
	+ \frac{m-1 - 4\epsilon}{2} \int |\xi|^{2m} W^2 d\xi\\
	&\leq C_\epsilon \int W^2 d\xi 
		+ C_\epsilon \|U\|_\infty^{2m} \int W^2 d\xi
		+ \frac{\|D\|_\infty}{2} \int |\xi|^{2m} W^2 d\xi
.\end{split}\]
We know that $\|D\|_{L^\infty}$ decays to zero, and there is sufficient control over $\|W\|_{L^2}$ and $\|U\|_{L^\infty}$ by lemma~\ref{e_vorticity_decay} and proposition~\ref{kernel_bounds}.  Hence choosing $\epsilon>0$ sufficiently small and integrating the above inequality yields the apriori estimate required in $L^2(m)^2$.  These a priori estimates are summarized in the following proposition.
\begin{proposition}
Fix $(W_0, D_0)\in X$ where $X$ is either $L^2(m)^2$, with $m>1$ and $\int D_0 d\xi = 0$, or $(L^2_w)^2$.  Then there exists a unique solution to the system~\sysWD which satisfies
\[
	\|W(\tau)\|_{X} \leq C
.\]
Here $C$ is a constant depending only on the initial data and which tends to zero as $\|W_0\|_{X}$ tends to zero.
\end{proposition}

\subsection*{An Iterative Scheme}

To prove existence and uniqueness of classical solutions with initial data in $L^1$ we follow~\cite{bblBen-Artzi94}.  For existence, we begin with smooth initial data,  and use an iterative argument to obtain the existence of solutions which are bounded in $L^p$ for every $p$.  The key contribution here is that we iterate only in the vorticity, leaving the divergence fixed as solutions to the heat equation follow from the classical theory.  We define $\omega_0=0$ and then let $\omega_k$ be the solution to the linear system
\[\begin{split}
	\del_t \omega_k + \nabla\cdot(u_{k-1} \omega_k) &= \Delta \omega_k\\
	u_k &= \gradinv d + \BS \omega_k.
\end{split}\]
Bounds similar to lemma~\ref{e_vorticity_decay} can be obtained for this system, establishing the existence of a solution.  Uniqueness follows by directly estimating the difference of two solutions.  Afterwards, a continuity argument is used to extend this to any initial data in $L^1$. % In addition, uniqueness follows by directly estimating the difference of two solutions.

In general, this argument differs from that in~\cite{bblBen-Artzi94} only in the appearance of an extra term involving $d$ in several of the estimates.  However, this extra term behaves much better than the non-linear term as the classical theory on the heat equation for $d$ yields appropriate bounds on the divergence in any of the required spaces.  In particular, this gives us the following result which we state without proof.

\begin{proposition}
Suppose that $\omega_0$ and $d_0$ are elements of $L^1(\R^2)$.  Then there exist $\omega, d \in C(\R_+, L^1) \cap C(\R_+, W^{1,1}\cap W^{1,\infty})$, where $\R_+ := (0,\infty)$, which are the unique solutions to the system~\sysOmegaD.
\end{proposition}

%
%
%\subsection*{Well-posedness in Weighted Spaces}%$L^2(m)$}
%
%In~\cites{bblGallayWayne02,bblRodrigues09}, the authors give detailed arguments proving the well-posedness to the system~\sysWD in $L^2(m)$ and $L^2_w$.  In order to reproduce these arguments in our setting, we simply need to obtain bounds in these spaces on the functions and their derivatives.  The strategy of Gallay and Wayne and Rodrigues involves simply integrating by parts with the appropriate weighting factor.  While extra terms will appear in our system, these will not be difficult to deal with.   Indeed, any new terms in our setting will involve divergence terms.  Since $D$ converges quickly to $\beta G$, these terms will be straightforward to estimate.
%%  the system~\sysWD by following the work of~\cite{bblGallayWayne02}*{theorem~3.2}.  Here, the authors show that if $W_0\in L^2(m)$, with $m > 1$, then a unique solution exists which remans bounded.  In our setting, their strategy is to multiply~\eqref{eqnW} by $G^{-1}W$ and to integrate by parts.  Here, we find ourselves in the same situation as in section~\ref{sxnWellPosed}. \sidenote{\textbf{[Landon: Here, ``above'' is ambiguous.][Chris: changed ``above'' to ``section 2'']}}  Extra terms appear involving the $D$ term; however, since we have better estimates on $D$, these are easy to deal with.
%In particular, this gives us the following result, which we state without proof.
%
%
%
%
%
%
%
%
%

  \section*{Acknowledgements}

  The authors thank Thierry Gallay for suggesting the problem to us and for many helpful discussions.

  \bibliographystyle{abbrv}%{{{1
  \bibliography{refs}
  %}}}1
\end{document}